%% file: gas_network.tex
\documentclass[10pt,a4paper]{article}
\input{marcos.tex}

\begin{document}
\title{Efficient Numerical Methods for Gas Network Modeling and Simulation\thanks{This work is funded by the European Regional Development Fund (ERDF/EFRE: ZS/2016/04/78156) within the Research Center Dynamic Systems: Systems Engineering (CDS).}}
\author{Yue Qiu\thanks{Corresponding author. Max Planck Institute for Dynamics of Complex Technical Systems, SandtorStra{\ss}e 1, 39108, Magdeburg, Germany. E-mail: qiu@mpi-magdeburg.mpg.de, y.qiu@gmx.us}\enspace \href{https://orcid.org/0000-0003-0360-0442}{\includegraphics[scale=0.1]{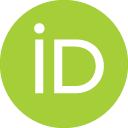}}, Sara Grundel\thanks{Max Planck Institute for Dynamics of Complex Technical Systems, SandtorStra{\ss}e 1, 39108, Magdeburg, Germany. E-mail: grundel@mpi-magdeburg.mpg.de}\enspace
\href{https://orcid.org/0000-0002-0209-6566}{\includegraphics[scale=0.1]{figures/orcid.png}}, Martin Stoll\thanks{Technische Universit{\"a}t Chemnitz, Faculty of Mathematics, Reichenhainer Stra{\ss}e 41, 09107
Chemnitz, Germany. E-mail: martin.stoll@mathematik.tu-chemnitz.de}, Peter Benner\thanks{Max Planck Institute for Dynamics of Complex Technical Systems, SandtorStra{\ss}e 1, 39108, Magdeburg, Germany. E-mail: benner@mpi-magdeburg.mpg.de}\enspace     \href{https://orcid.org/0000-0003-3362-4103}{\includegraphics[scale=0.1]{figures/orcid.png}}
}
\date{~}
\maketitle

\begin{abstract}
    We study the modeling and simulation of gas pipeline networks, with a focus on fast numerical methods for the simulation of transient dynamics. The obtained mathematical model of the underlying network is represented by a nonlinear differential algebraic equation (DAE).
    {With our modeling, we reduce the number of algebraic constraints, which correspond to the $(2,2)$ block in our semi-explicit DAE model, to the order of junction nodes in the network, where a junction node couples at least three pipelines. We can furthermore ensure that the $(1, 1)$ block of all system matrices including the Jacobian is block lower triangular by using a specific ordering of the pipes of the network.} We then exploit this structure to propose an efficient preconditioner for the fast simulation of the network. We test our numerical methods on benchmark problems of (well-)known gas networks and the numerical results show the efficiency of our methods.
    
\vspace{0.3cm}
\end{abstract}
\setcounter{section}{0}

\section{Introduction}
\input{introduction}

\section{Gas Dynamics in Pipelines}\label{sec_euler}
\input{euler}

\section{Network Modeling}\label{sec_net}
\input{network}

\section{Fast Numerical Methods for Simulation}\label{sec_numerical}
\input{numerical}

\section{Numerical Results}\label{sec_results}
\input{results}

\section{Conclusions}
\input{conclusions}

\bibliographystyle{unsrt}
\bibliography{refer}
\end{document}

%% file: marcos.tex
\usepackage{cite}
\usepackage[british]{babel}
\usepackage{mathrsfs,graphicx,color,float, indentfirst,textcomp}
\usepackage{setspace}
\usepackage{latexsym,amssymb,lscape,amsmath,amsthm,amsfonts,amssymb,cite,enumerate}
\usepackage{lscape}
\usepackage[linktocpage=true, colorlinks=true, linkcolor=blue, citecolor=red, urlcolor=magenta]{hyperref}
\usepackage{subfigure}
\usepackage{algorithm} 
\usepackage{algorithmic} 
\usepackage{multirow} 
\usepackage{pifont}
\usepackage{doi}

\usepackage{xfrac}
\usepackage{indentfirst}
\usepackage{algorithm} 
\usepackage{algorithmic}
\usepackage{multirow} 
\usepackage{xcolor}
\usepackage{geometry}
\usepackage{booktabs}

\usepackage{tikz}
\usepackage{pgfplots}
\pgfplotsset{compat=newest}

\newtheorem{theorem}{Theorem}[section]
\newtheorem{lemma}[theorem]{Lemma}
\newtheorem{remark}{Remark}[section]
\newtheorem{definition}{Definition}[section]

\newtheorem{proposition}[theorem]{Proposition}

%% file: introduction.tex
Natural gas is one of the most widely used energy sources in the world, as it is easily transportable, storable and usable to generate heat and electricity. Even though research on the transient gas network dates back to the 1980s \cite{Osiadacz1984, Osiadacz1987}, often only  stationary solutions of the gas network are computed. This is also reasonable as the variation in a classically operated gas transportation networks enforces no need for a truly transient simulation. However as we move from classical energy sources to renewable energy sources in which we may use the gas pipelines to deal with flexibility from volatile energy creation, the need for fast transient simulation will increase. In recent years, research on natural gas networks focuses on a variety of topics: transient simulations~\cite{Osiadacz1984, Osiadacz1987, Tao1998, GruHR16, Gugat2015}, optimization and control~\cite{Steinbach2007, Zlotnik2015, Hante2017}, time splitting schemes for solving the parabolic flow equations~\cite{Zhou2000}, discretization methods~\cite{Egger2018, Osiadacz1989}, and model sensitivity study~\cite{Chaczykowski2009} to mention a few. It is obvious that efficient simulation techniques are needed both for design and for control.

{The objective of this paper is to speed up the computations at the heart of each simulation. For that we make use of a discretization that respects the hyperbolic nature of the problem by using a finite volume method (FVM), as well as exploiting the network structure to create good properties for the computations. We start as it is standard for modeling gas transport in a pipeline by the one-dimensional isothermal Euler equation, which is a partial differential equation
    (PDE).} {We introduce the necessary discrete variables for the pressure and the flux on each pipe and make sure to use as little algebraic equations as possible when considering a network of pipes.
        By further exploiting the structure of the system, we propose a preconditioner that enables fast solution of such a nonlinear equation using a preconditioned Krylov solver at each Newton iteration.}


The structure of this paper is as follows. We introduce the incompressible isothermal Euler equation for the gas dynamics modeling of each pipeline of the network in Section~\ref{sec_euler}, and we apply the finite volume method (FVM) to discretize the incompressible isothermal Euler equation in Section~\ref{sec_euler}. In Section~\ref{sec_net}, we introduce the details of gas network modeling starting from assembling all pipelines. This results in a set of nonlinear DAEs for the network model. We propose numerical algorithms to solve the resulting nonlinear DAE in  Section~\ref{sec_numerical} to simulate the gas network. We use benchmark problems from gas pipeline networks to show the efficiency and the advantage of our numerical algorithms in Section~\ref{sec_results}, and we draw conclusions in the last section.

%% file: euler.tex

In  a typical gas transport network the main components are pipelines (or pipes, for short). In this section, we will discuss the dynamics of gas transported along pipes. 

\subsection{1D Isothermal Euler Equation}

The dynamics of gas transported along pipes is described by the Euler equation, which represents the laws of mass conservation, momentum conservation, and energy conservation. In this paper, we assume that the temperature is constant throughout the gas network, leading to the isothermal Euler equations. Therefore, the energy equation can be neglected. This may seem unrealistic, but for onshore gas networks, in which the pipes are buried underground, the temperature along pipes does not change much. This assumption
greatly reduces the complexity of modeling and is widely used in the simulation of gas networks~\cite{Herran-Gonzalez2009,Herty2010, GruHKetal13,Gugat2015,Fuegenschuh2015}. 


Consider the 1D isothermal Euler equation over the spatial domain $[0,\ L]$ given by
\begin{subequations}
\begin{align}
    \frac{\partial}{\partial t}\rho &= -\frac{\partial}{\partial x} \varphi, \label{eqn:mass} \\
      \frac{\partial}{\partial t}\varphi &= - \frac{\partial}{\partial x} p - \frac{\partial}{\partial x} (\rho v^2) - g \rho \frac{\partial}{\partial x} h - \frac{\lambda(\varphi)}{2d} \rho v |v|, \label{eqn:momentum} \\
      p &= \gamma(T)z(p, T)\rho.\label{eqn:state}
\end{align}
\end{subequations}
Here, $\rho$ is the density of the gas ($kg/m^3$), $\varphi$ represents the flow rate and $\varphi = \rho v$ with $v$ the velocity of the gas ($m/s$), $d$ is the diameter of the pipe ($m$), $\lambda$ is the friction factor of the gas. Meanwhile, $p$ denotes the pressure of the gas
($N/m^2$), $T$ is the temperature of the gas ($K$), and $z$ denotes the compressibility factor. The conservation of mass is given by~\eqref{eqn:mass}, and the conservation of momentum is represented by~\eqref{eqn:momentum}, while the state equation~\eqref{eqn:state} couples the pressure with the density. 

By using the mass flow $q=a\varphi$ to substitute into~\eqref{eqn:mass}--\eqref{eqn:momentum}, where $a$ is the cross-section area of pipes, we get

\begin{subequations}
    \begin{align}
         \frac{\partial}{\partial t}\rho &= -\frac{1}{a} \frac{\partial}{\partial x} q, \label{eqn:mass_2} \\
          \frac{1}{a}\frac{\partial}{\partial t}q &= -\frac{\partial}{\partial x} p - \frac{1}{a^2}\frac{\partial}{\partial x} \frac{q^2}{\rho} - g \rho \frac{\partial}{\partial x} h - \frac{\lambda(q)}{2da^2} \frac{q|q|}{\rho}, \label{eqn:momentum_2} \\
           p &= \gamma(T)z(p,T)\rho. \label{eqn:state_2}
    \end{align}
\end{subequations}
For the isothermal case, the temperature $T$ equals $T_0$ throughout the network, then $\gamma(T) = \gamma(T_0) = \gamma_0$, and $z(p, T) = z(p, T_0) = z_0(p)$. Therefore, the compressibility factor $z(p, T)$ is only related to the pressure $p$ and we can rewrite~\eqref{eqn:mass_2}--\eqref{eqn:state_2} as
\begin{subequations}
    \begin{align}
        \frac{1}{\gamma_0} \frac{\partial}{\partial t}\frac{p}{z_0(p)} &= - \frac{1}{a} \frac{\partial}{\partial x} q,\label{eqn:mass_3} \\
          \frac{1}{a}\frac{\partial}{\partial t}q &= -\frac{\partial}{\partial x} p - \underbrace{\frac{\gamma_0}{a^2}\frac{\partial}{\partial x} \frac{q^2 z_0(p)}{p}}_{\text{inertia\ term}} - \underbrace{\frac{g}{\gamma_0}\frac{p}{z_0(p)} \frac{\partial}{\partial x} h}_{\text{gravity\ term}} - \underbrace{\frac{\lambda(q)\gamma_0}{2da^2} z_0(p) \frac{q|q|}{p}}_{\text{friction\ term}}.\label{eqn:momentum_3}
    \end{align}
\end{subequations}
For the inertia term, 
it is studied in~\cite{Herran-Gonzalez2009} that 
\[
    \frac{\gamma_0}{a^2}\frac{\partial}{\partial x}\frac{q^2z_0(p)}{p}\approx 10^{-3}\frac{\partial}{\partial x}p.
\]
Therefore, the inertia term can be neglected and this neglection greatly simplifies the model, which is standard in the study of gas networks~\cite{Hante2017,GruHR16,Zlotnik2015}. In this paper, we also use this simplification. Meanwhile, we will often assume that the elevation of pipes is homogeneous. The gravity term in~\eqref{eqn:momentum_3} then vanishes. However this term is easily treatable within our framework, which we will illustrate later in this section.

Now, we get the model that describes the dynamics of isothermal gas transported along homogeneous elevation pipes given by
\begin{subequations}
    \begin{align}
        \frac{\partial}{\partial t}\frac{p}{z_0(p)} &= - \frac{\gamma_0}{a} \frac{\partial}{\partial x} q,\label{eqn:mass_4}\\
        \frac{\partial}{\partial t}q &= -a\frac{\partial}{\partial x} p - \frac{\lambda(q)\gamma_0}{2da} z_0(p) \frac{q|q|}{p}.\label{eqn:momentum_4}
    \end{align}
\end{subequations}

The details of modeling the compressibility factor $z_0(p)$ and the friction factor $\lambda(q)$ are described in~\cite{BenGHetal17}.

\subsection{Finite Volume Discretization}\label{sec_fvm}
In this paper, the dynamics of the gas transported along pipes are described by the 1D isothermal incompressible Euler equation ($z_0(p)=1$) over the spatial domain $[0,\ L]$ with homogeneous elevation. According to~\eqref{eqn:mass_4}--\eqref{eqn:momentum_4} we have
\begin{subequations}\label{eqn:iso_euler}
    \begin{align}
        \frac{\partial}{\partial t}p + \frac{c}{a} \frac{\partial}{\partial x} q &=0,\label{eqn:mass_5}\\
        \frac{\partial}{\partial t}q +a\frac{\partial}{\partial x} p + \frac{c\lambda}{2da} \frac{q|q|}{p} &=0.\label{eqn:momentum_5}
    \end{align}
\end{subequations}
For simplification of notation we introduce $c=\gamma_0$ and we also assume $\lambda(q)\equiv\lambda$. The system~\eqref{eqn:mass_5}--\eqref{eqn:momentum_5} is nonlinear due to the friction term. For gas transportation pipes, the boundary condition at the inflow point $x=0$ is given by the prescribed pressure $p_s$, while the boundary condition at the outflow point $x=L$ is represented by the given mass flow (gas demand) $q_d$. Therefore, the boundary conditions for~\eqref{eqn:mass_5}--\eqref{eqn:momentum_5}
are
given as
\begin{equation}\label{eqn:bcs}
    \begin{cases}
        p = p_s, & \text{at}\ x=0,\\
        q = q_d, & \text{at}\ x=L.
    \end{cases}
\end{equation}
For the well-posedness and the regularity of the solution of the system~\eqref{eqn:iso_euler}--\eqref{eqn:bcs}, we refer to~\cite{EggKS17}. 


{More advanced numerical schemes for hyperbolic PDEs, such as the total variation diminishing (TVD) method~\cite{TorB00}, or the discontinuous Galerkin method (DG)~\cite{JohP86}, could be implemented at the next step of our research to investigate more complicated dynamics of the gas networks. The scope of our paper is to develop a systematic numerical methodology for the fast simulation of the network dynamics while taking numerical accuracy into account.}

For such an FVM discretization, we partition the domain as shown in Figure~\ref{fig:1d_part}.

\begin{figure}[H]
\centering
 \includegraphics[width=0.3\textwidth]{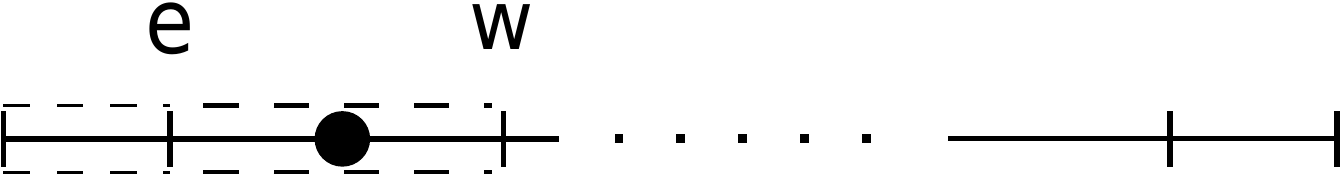}
 \caption{finite volume cells partition}\label{fig:1d_part}
\end{figure}

In Figure~\ref{fig:1d_part}, the left boundary of the control volume is denoted by `e' while the right boundary of the control volume is denoted by `w'. For the isothermal incompressible Euler equation~\eqref{eqn:mass_5}--\eqref{eqn:momentum_5}, we have two variables, i.e., the pressure $p$, and the mass flow $q$. Together with the boundary condition~\eqref{eqn:bcs}, we use two different control volume partitions for the pressure and mass flow nodes, which are shown in
Figure~\ref{fig:1d_p_q}.

\begin{figure}[H]
\centering
\subfigure[control volumes partition for $p$]
    {\label{fig:1d_p}
  \includegraphics[width=0.3\textwidth]{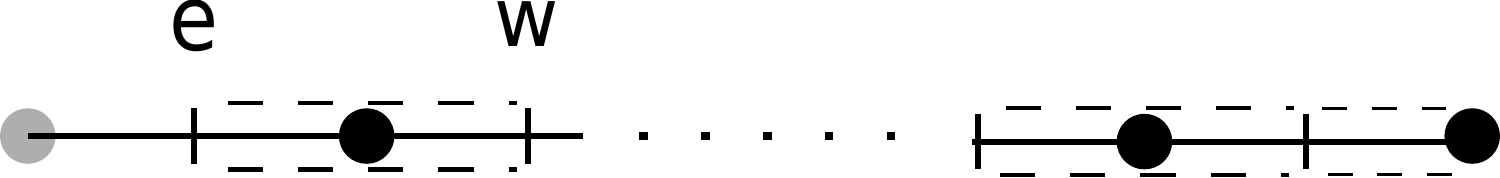}}
\qquad  \qquad
\subfigure[control volumes partition for $q$]
    {\label{fig:1d_q}
  \includegraphics[width=0.3\textwidth]{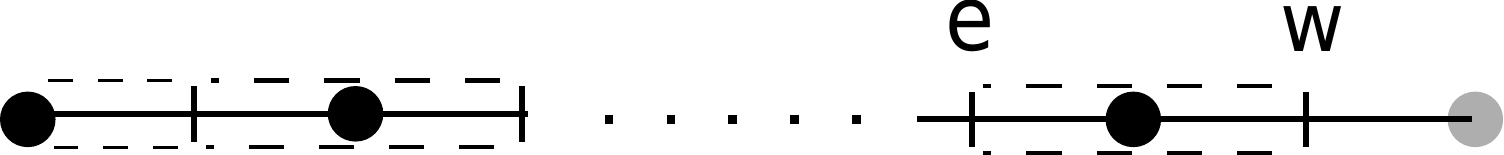}}
\caption{control volumes partition for $p$ and $q$}\label{fig:1d_p_q}
\end{figure}

To apply the FVM to discretize the PDE and the boundary condition, we integrate~\eqref{eqn:mass_5}--\eqref{eqn:momentum_5} over each control volume. To be specific, we integrate~\eqref{eqn:mass_5} over the pressure control volume in Figure~\ref{fig:1d_p}, and integrate~\eqref{eqn:momentum_5} over the mass flow control volume in Figure~\ref{fig:1d_q}.

To integrate~\eqref{eqn:mass_5} over the $i$-th control pressure control volume $C_i$, we have
\[
    \int_{C_i}\left(\partial_t p + \frac{c}{a}\partial_x q\right)\ dx  = 0.
\]
The discretization point in $C_i$ is either a virtual node along a pipe or a real node that connects two different pipes. Therefore, the coefficient $a$ of the PDE~\eqref{eqn:mass_5}--\eqref{eqn:momentum_5}, which represents the cross-section area of a pipe, may have a sudden change at the node in control volume $C_i$. Here, we use $C_i^{-}$ and $C_i^{+}$ to partition the control volume $C_i$ with $C_i=C_i^-\cup C_i^+$, and the lengths of $C_i^-$ and $C_i^+$ are $\frac{h_{i-1}}{2}$ and
$\frac{h_i}{2}$, respectively. This partition is shown in Figure~\ref{fig:ci_part}. 
\begin{figure}[H]
\centering
 \includegraphics[width=0.25\textwidth]{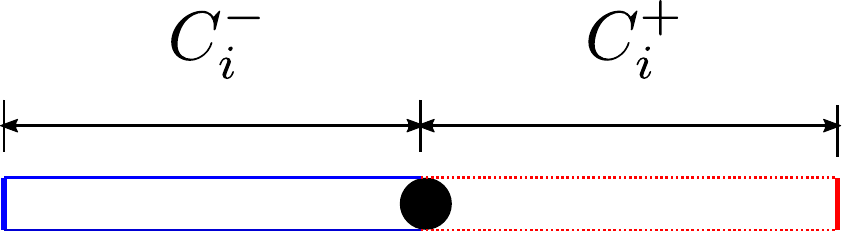}
 \caption{Separation of control volume $C_i$}\label{fig:ci_part}
\end{figure}

Therefore, we get
\[
    \int_{C_i}\partial_t p\ dx +\int_{C_i^{-}}\frac{c}{a}\partial_x q\ dx  + \int_{C_i^{+}}\frac{c}{a}\partial_x q\ dx = 0.
\]
Furthermore, by applying the midpoint rule, we get
\[
    \frac{h_{i-1}+h_i}{2}\partial_t p_i + \frac{c}{a_{i-1}}\left(q_i - \frac{q_{i-1}+q_i}{2}\right) + \frac{c}{a_{i}}\left(\frac{q_{i}+q_{i+1}}{2}-q_i\right) = 0, 
\]
i.e.,
\begin{equation}\label{eqn_int_com_het}
 \frac{h_{i-1}+h_i}{2} \partial_t p_i + c \left(-\frac{1}{2a_{i-1}}q_{i-1} + \left(\frac{1}{2a_{i-1}}-\frac{1}{2a_i}\right)q_i + \frac{1}{2a_i}q_{i+1}\right) = 0,
\end{equation}
where $p_i$ and $q_i$ are the values at the center of the control volume for $i=2,\dots,n-1$.
Similarly, for the rightmost control volume of the pressure in Figure~\ref{fig:1d_p}, we have
\begin{equation}\label{eqn_rb_com_het}
        \int_{C_n^-}\left(\partial_t p + \frac{c}{a}\partial_x q\right)\ dx \approx 
        \frac{h_{n-1}}{8}\partial_t p_{n-1} + \frac{3h_{n-1}}{8}\partial_t p_n + \frac{c}{2a_{n-1}}(q_n - q_{n-1}) = 0,
\end{equation}
where $p_n$ and $q_n$ are the values at the end of the pipe.
For the $i$-th mass flow control volume $C_i$ shown in Figure~\ref{fig:1d_q}, we integrate~\eqref{eqn:momentum_5} over it and get
\[
    \int_{C_i} \left(\partial_t q + a\partial_x p\right)\ dx  =  -\frac{c}{2}\int_{C_i} \frac{\lambda}{ad}\frac{q|q|}{p}\ dx.
\]
Applying the same partition of $C_i$ as in Figure~\ref{fig:ci_part}, we have,
\[
 \begin{split}
     \int_{C_i} \left(\partial_t q + a\partial_x p\right)\ dx & = \int_{C_i^{-}} \left(\partial_t q + a\partial_x p \right)\ dx + \int_{C_i^{+}} \left(\partial_t q + a\partial_x p \right)\ dx \\
  & \approx \frac{h_{i-1}}{2} \partial_t q_i + \frac{a_{i-1}}{2}(p_i - p_{i-1}) + \frac{h_{i}}{2} \partial_t q_i + \frac{a_{i}}{2}(p_{i+1} - p_{i}) \\
   & = \frac{h_{i-1}+h_i}{2} \partial_t q_i + \left(-\frac{a_{i-1}}{2}p_{i-1} + \frac{a_{i-1}-a_i}{2}p_i + \frac{a_i}{2}p_{i+1}\right),
 \end{split}
\]
and
\[
    \int_{C_i}\frac{\lambda}{ad}\frac{q|q|}{p}\ dx \approx \frac{1}{2}\left(\frac{h_{i-1}\lambda_{i-1}}{a_{i-1}d_{i-1}}+\frac{h_{i}\lambda_{i}}{a_{i}d_{i}}\right)\frac{q_i|q_i|}{p_i}.
\]
Therefore, for the $i$-th mass flow control volume $C_i$, we have
\begin{equation}\label{eqn_int_mom_het}
 \frac{h_{i-1}+h_i}{2} \partial_t q_i + \left(-\frac{a_{i-1}}{2}p_{i-1} + \frac{a_{i-1}-a_i}{2}p_i + \frac{a_i}{2}p_{i+1}\right)=-\frac{c}{4}\left(\frac{h_{i-1}\lambda_{i-1}}{a_{i-1}d_{i-1}}+\frac{h_{i}\lambda_{i}}{a_{i}d_{i}}\right)\frac{q_i|q_i|}{p_i}.
\end{equation}
Similarly, for the leftmost mass flow control volume, we get
\begin{equation}\label{eqn_lb_mom_het}
    \frac{3h_1}{8}\partial_t q_1 +\frac{h_2}{8}\partial_t q_2 + \frac{a_1}{2}(-p_1 + p_2) = -\frac{ch_1\lambda_1}{4a_1d_1}\frac{q_1|q_1|}{p_1}.
\end{equation}
where $p_1$ and $q_1$ are the values at the beginning of the pipe.
Associating~\eqref{eqn_int_com_het}--\eqref{eqn_lb_mom_het} with the boundary condition~\eqref{eqn:bcs}, we get
\begin{equation}\label{eqn:dis_sys}
    \underbrace{\begin{bmatrix}
  M_p & \\
      & M_q
    \end{bmatrix}}_{\mathcal{M}}
\begin{bmatrix}
 \partial_t p \\
 \partial_t q
\end{bmatrix}
=
    \underbrace{\begin{bmatrix}
 0 & K_{pq} \\
 K_{qp} & 0
    \end{bmatrix}}_{\mathcal{K}}
\begin{bmatrix}
 p \\
 q
\end{bmatrix}
+
    \overbrace{\underbrace{\begin{bmatrix}
 B_q \\
 0
    \end{bmatrix}}_{\mathcal{B}_q}
q_d}^{\text{right BC}}
+
    \overbrace{\underbrace{\begin{bmatrix}
 0 \\
 B_p
    \end{bmatrix}}_{\mathcal{B}_p}
p_s}^{\text{left BC}}+
\begin{bmatrix}
 0 \\
 g(p_s,p, q)
\end{bmatrix},
\end{equation}
where the mass matrices $M_p$ and $M_q$ are given by
\[
 M_p =  \begin{bmatrix}
	  \frac{h_1+h_2}{2} & 	& 	 &   	         &  \\
	                    & \frac{h_2+h_3}{2} & 	 & 	         & \\
	    &   & \ddots &               &  \\
	    &   &        &   \frac{h_{n-2}+h_{n-1}}{2}           &   \\
        &   &	 &  \frac{h_{n-1}}{8}  & \frac{3h_{n-1}}{8}
       \end{bmatrix},\quad M_q = \begin{bmatrix}
           \frac{3h_1}{8} & \frac{h_1}{8}	  &  	&   	&  \\
				             &  \frac{h_1+h_2}{2}            & 	        &     	& \\
	                                     &  	     &  \frac{h_2+h_3}{2}	&       &  \\
	                                     &  	     &          & \ddots     &   \\
	                                     &		     &		&  	& \frac{h_{n-2}+h_{n-1}}{2}
			       \end{bmatrix},
\]
and
\[
 K_{pq} = -\frac{c}{2}\begin{bmatrix}
             -\frac{1}{a_1} &  \frac{1}{a_1}-\frac{1}{a_2} & \frac{1}{a_2}                       & 	                   &                                   & \\
			     & -\frac{1}{a_2}                & \frac{1}{a_2}-\frac{1}{a_3}	     &  \frac{1}{a_3}     &        & \\
	       &    & \ddots & \ddots & \ddots &  \\
	       &    &        & -\frac{1}{a_{n-3}}     & \frac{1}{a_{n-3}}-\frac{1}{a_{n-2}}      & \frac{1}{a_{n-2}} \\
	       &    &        &        &  -\frac{1}{a_{n-2}}    & \frac{1}{a_{n-2}}-\frac{1}{a_{n-1}} \\
	       &    &        &        &        & -\frac{1}{a_{n-1}}
            \end{bmatrix},
\]
is an upper-triangular matrix with 3 diagonals and
\[
 K_{qp} = -\frac{1}{2}\begin{bmatrix}
             a_1      &           &          &          &        & \\
	     a_1-a_2  &  a_2      &          &          &        & \\
	     -a_2     &  a_2-a_3  & a_3      &          &        &  \\
	              & -a_3      & a_3-a_4  & a_4      &        &  \\
	              &           &   \ddots &  \ddots  & \ddots &  \\
	              &           &          & -a_{n-2} & a_{n-2}-a_{n-1} & a_{n-1}
          \end{bmatrix},
\]
is a lower-triangular matrix with 3 diagonals. Meanwhile,
\begin{equation}
 B_q = -\frac{c}{2}\begin{bmatrix}
		    0 \\
		    \vdots \\
		    0 \\
		    \frac{1}{a_{n-1}}\\
		    \frac{1}{a_{n-1}}
         \end{bmatrix},\quad
         B_p = \frac{1}{2}\begin{bmatrix}
                           a_1 \\
                           a_1 \\
                           0 \\
                           \vdots \\
                           0
                          \end{bmatrix},\quad
                          g(p_s,p, q) = -\frac{c}{4}\begin{bmatrix}
                                                 \frac{h_1\lambda_1}{a_1d_1}\frac{q_1|q_1|}{p_s} \\
                                                 (\frac{h_1\lambda_1}{a_1d_1} + \frac{h_2\lambda_2}{a_2d_2})\frac{q_2|q_2|}{p_2} \\
                                                 (\frac{h_2\lambda_2}{a_2d_2} + \frac{h_3\lambda_3}{a_3d_3})\frac{q_3|q_3|}{p_3} \\
                                                 \vdots \\
                                                 (\frac{h_{n-2}\lambda_{n-2}}{a_{n-2}d_{n-2}} + \frac{h_{n-1}\lambda_{n-1}}{a_{n-1}d_{n-1}})\frac{q_{n-1}|q_{n-1}|}{p_{n-1}}
                                                \end{bmatrix}.\label{eq:gfunction}
\end{equation}
The vectors
\[
 p =\begin{bmatrix}
     p_2 & p_3 & \cdots & p_n
    \end{bmatrix}^T,\quad  
    q = \begin{bmatrix}
         q_1 & q_2 & \cdots & q_{n-1}
        \end{bmatrix}^T,
\]
represent the discretized analog of $p$ and $q$ to be computed, while $p_1=p_s$, and $q_n=q_d$.

The discretized model for the dynamics of gas transported along pipes shown in~\eqref{eqn:dis_sys} is a nonlinear ordinary differential equation (ODE). The nonlinear term in this ODE comes from the discretization of the friction term in the momentum equation~\eqref{eqn:momentum_5}. 

\begin{remark}
    The obtained model in~\eqref{eqn:dis_sys} results from the discretization of the incompressible isothermal Euler equation of homogeneous elevation~\eqref{eqn:mass_5}--\eqref{eqn:momentum_5}. However, we note that for the heterogeneous elevation case, the gravity term in~\eqref{eqn:momentum_3} is linear in the pressure $p$. After discretization, this term will introduces an additional term in the position of $K_{qp}$ in~\eqref{eqn:dis_sys}. This new term does not change the structure of the
    model that we obtained in~\eqref{eqn:dis_sys}. The structure we refer to here is the block structure of the matrices involved in describing the equation as well as the sparsity pattern of all Jacobians of the nonlinear functions.
\end{remark}

Note that we impose the boundary condition~\eqref{eqn:bcs} by using the prescribed pressure at the inflow point and the given mass flow at the outflow point. However one can also impose a negative mass flow at the outflow point making it an inflow point, and for more complex topology of the network with more than one supply node, it can happen that the gas flows out at a so called inflow point.  Computational results in the numerical experiment show this.

%% file: network.tex
Within this paper, we focus on passive networks to demonstrate how advanced numerical linear algebra can benefit the fast simulation of such a network. When we say passive network we mean a network that  does not contain active elements, such as compressors, valves, etc.~\cite{GonHH17}. This simplification allows us to maintain a differential algebraic model without combinatorial aspects. For the modeling of the network with compressors and valves we refer
to~\cite{BenGHetal17,Her07}. In this section, we introduce a new scheme for the modeling of the passive networks.


The abstract gas network is described by a directed graph
\begin{equation}\label{eqn:dir_graph}
    \mathcal{G} = (\mathcal{E},\ \mathcal{N}),
\end{equation}
where $\mathcal{E}$ denotes the set of edges, which contains the pipes in the gas network. $\mathcal{N}$ represents the set of nodes, which consist of the set of supply nodes $\mathcal{N}_s$, demand nodes $\mathcal{N}_d$, and interior nodes $\mathcal{N}_0$ of the network. Here, the supply nodes represent the set of nodes in the network where gas is injected into the network or more precisely where the pressure is given, and the demand nodes form a set of nodes where the gas is extracted, where an outgoing flux is described and interior nodes are the rest. We assume from now on that demand nodes and supply nodes are the only boundary nodes. That means they are only connected to one pipe. If supply or demand nodes exist that are connected to more than one edge, we add a short pipe to that node and declare the new node as the demand or supply node and the old one becomes an interior node. Sometimes interior nodes are called junction nodes~\cite{GruHJetal14}, but for us junction nodes are more specific.

\begin{definition}\label{def:junction}
The nodes inside a graph $\mathcal{G}$, which connect at least three edges, are called junction nodes. 
\end{definition}

Denote $N_j\subset \mathcal{N}$ the set of junction nodes of a given graph $\mathcal{G}$. Since in our graph $\mathcal{N}_s\cup\mathcal{N}_d$ is equal to the set of boundary nodes, we have $\mathcal{N}_j\subset\mathcal{N}_0$. From now on we will assume that besides our original graph $\mathcal{G}$, we also have the graph $\tilde{\mathcal{G}}$, which is the graph created from $\mathcal{G}$ by smoothing out the vertices $\mathcal{N}_0\backslash\mathcal{N}_j$. Smoothing the vertex $w$, which is
connected to the edges $e_1$ and $e_2$, is the operation which removes $w$ as well as both edges and adds a new edge to the starting and end points of the pair. This edge can be given the direction of any of the two removed edges. Here, it is emphasized that only vertices that connect exactly two edges can be smoothed. This is however the case for the vertices in $\mathcal{N}_0\backslash\mathcal{N}_j$. This means our new graph $\tilde{\mathcal{G}}= (\tilde{\mathcal{E}},\ \tilde{\mathcal{N}})$
is a directed graph, which only has demand nodes, supply nodes and junction nodes in the sense of Definition~\ref{def:junction}.
\begin{figure}[H]
\centering
    \includegraphics[width=0.25\textwidth]{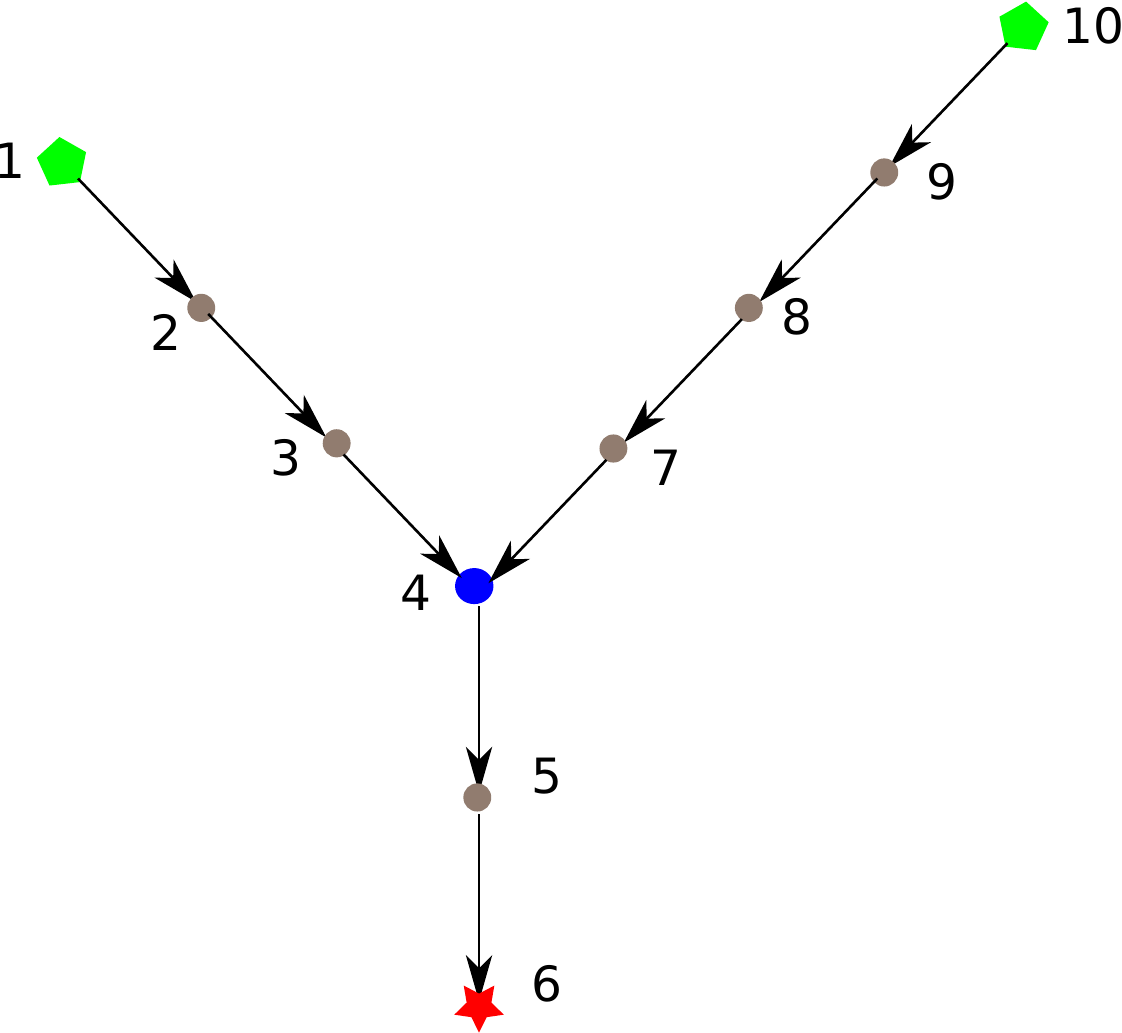}
    \caption{A typical gas network}\label{fig:forked}
\end{figure}

An example network is shown in Figure~\ref{fig:forked}. Here, nodes $1$ and $10$ denote the supply nodes, node $6$ represents the demand node. According to Definition~\ref{def:junction}, only node $4$ in Figure~\ref{fig:forked} is a 
junction node. This means we can replace this graph by the smoothed graph given in Figure~\ref{fig:junction}

\begin{figure}[H]
\centering
    \includegraphics[width=0.08\textwidth]{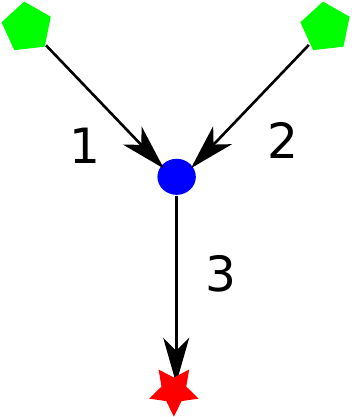}
    \caption{Smoothed network of Figure~\ref{fig:forked} with an ordering of the pipes.}\label{fig:junction}
\end{figure}

{By making use of the smoothed graph, the number of algebraic constraints is kept small, while the edges which represent several pipes become long. More details on the number of algebraic constraints will be presented in Proposition~\ref{prop:alg} at the end of this section. For each such long pipe $i\in\tilde{\mathcal{E}}$, we have the set of variables $p^{(i)}_1,\dots,p^{(i)}_{n^{(i)}}$ and $q^{(i)}_1,\dots,q^{(i)}_{n^{(i)}}$ that represent the discrete analog of $p$
and $q$, respectively. Here $n^{(i)}$ is the number of discretization points in a certain edge. The smoothed out nodes of the original graph are merely discretization points. We will denote $p^{(i)}_1=p^{(i)}_s$ and $q^{(i)}_{n^{(i)}}=q^{(i)}_d$, for now, as they are boundary conditions in a one pipe system.}

\begin{remark}\label{rmk:graph}Our gas network is modeled by a directed graph $\mathcal{G} = (\mathcal{E},\ \mathcal{N})$, where the set of boundary nodes is equal to the union of supply nodes and demand nodes. All edges $\mathcal{E}$ are pipes, and an edge attached to a supply node is called a supply pipe, while an edge attached to a demand node is called a demand pipe. A supply pipe is directed away from the supply node and a demand pipe is directed towards the demand node. By smoothing
of this graph as explained above, we also get a directed graph $\tilde{\mathcal{G}}= (\tilde{\mathcal{E}},\ \tilde{\mathcal{N}})$, whose interior nodes are all junction nodes in the sense of Definition~\ref{def:junction}. In $\tilde{\mathcal{G}}$, supply edges and demand edges still exist, and they should be directed as above, but are possibly longer. \end{remark}

\subsection{Nodal Conditions}\label{sub_sec:nodal}
There are two types of constraints concerning the connection of edges, namely the pressure constraint, and the mass flow constraint. These two types of constraints represent the equality of the dynamic pressure and conservation of mass at junction nodes~\cite{Her08}.

The so-called pressure nodal condition describes the pressure equality among pipes connected at a same junction node, which is given by
\begin{equation}\label{eqn:p_nodal}
p^{(i)}_{n^{(i)}} = p^{(j)}_s, \mbox{if a node connects the incoming pipe $i$ and the outgoing pipe $j$.}
\end{equation}
The pressure nodal condition states that the pressure at the end of the outflow pipes should equal the pressure at the beginning of the inflow pipes that connect to the same junction node and ensures that the there is only one pressure value at each node.

The second type of nodal condition i.e., the mass flow nodal condition, states the conservation of mass flow at the junction nodes, and it is given by
\begin{equation}\label{eqn:q_nodal}
\sum_{i\in \mathcal{I}_k} q_d^{(i)} = \sum_{i\in\mathcal{O}_k} q_1^{(i)} \mbox{ for every node $k$,}
\end{equation}
where $\mathcal{I}_k$ is the set of edges incoming the node $k$ and $\mathcal{O}_k$ the set of edges outgoing of node $k$. Equation~\eqref{eqn:q_nodal} states that the inflow at the junction node $k$ should equal to the outflow at the same junction node $k$. 

\subsection{Network Assembly}

In the discretized model~\eqref{eqn:dis_sys} describing the dynamics of gas transported along one single pipe the variables $p_1=p_s$ and $q_n=p_d$ are given by the boundary condition, i.e., the prescribed pressure at the input node, and the prescribed mass flow at the demand node. For the network all variables including $p_s^{(i)}$ and $q_d^{(i)}$ are treated as variables and we add the algebraic constraints ~\eqref{eqn:p_nodal}--\eqref{eqn:q_nodal} to the system. We take the network in Figure~\ref{fig:forked}, whose smoothed graph, with edge ordering is given by Figure~\ref{fig:junction} as an example to build the full DAE system. For each pipe we have the pipe dynamics of the discretized system given by \eqref{eqn:dis_sys}. The supply pressure for pipe 1 and 2 is given as well as the demand flux for pipe 3. This means we have the extra variables $p^{(3)}_s,q^{(1)}_d,q^{(2)}_d$. We will first use the fact that $p^{(3)}_s$ is equal to $p^{(1)}_{n^{(1)}}$ and replace it in the equation. We will then still have to make sure that $p^{(1)}_{n^{(1)}}=p^{(2)}_{n^{(2)}}$ and also that the incoming flux at the junction is equal to the outgoing flux. To summarize  $q^{(1)}_d,q^{(2)}_d$ are the added variables as $p_s^{(3)}$ was replaced directly and
 \begin{equation}
 \begin{split}
 q_d^{(1)}+q_d^{(2)}=q_1^{(3)},\\
p^{(1)}_{n^1}=p^{(2)}_{n^2},
 \end{split}
 \end{equation}
are the added algebraic constraints.
 
 By using the single pipe model~\eqref{eqn:dis_sys}, we obtain the mathematical model for the network in Figure~\ref{fig:forked} and \ref{fig:junction},
\begin{equation}\label{eqn:exm_daes}
    \begin{split}
        \underbrace{\begin{bmatrix}
               \mathcal{M}^{(1)}    &  & & & \\
                & \mathcal{M}^{(2)} &  & & \\
                &               & \mathcal{M}^{(3)} & & \\
                &   &   & 0 & \\
                &&      &   & 0
        \end{bmatrix}}_{\mathbf{M}}
       \frac{\partial}{\partial t}
       \begin{bmatrix}
       u^{(1)} \\
       u^{(2)} \\
       u^{(3)} \\
       q_d^{(1)} \\
       q_d^{(2)}
       \end{bmatrix}
       & =
        \underbrace{\begin{bmatrix}
       \mathcal{K}^{(1)}          &  & & \mathcal{B}_q^{(1)} & \\
       & \mathcal{K}^{(2)}        &  & & \mathcal{B}_q^{(2)}\\
       \bar{\mathcal{B}}_p^{(3)}  &  & \mathcal{K}^{(3)} & & \\
       &                      & e_3 & 1 & 1 \\
       e_1                    &e_2  &      &   & 0
        \end{bmatrix}}_{\mathbf{K}}
       \begin{bmatrix}                           
       u^{(1)} \\
       u^{(2)} \\
       u^{(3)} \\
     q_d^{(1)} \\
       q_d^{(2)}
  \end{bmatrix}\\
        & + \underbrace{\begin{bmatrix}
       \mathcal{B}_p^{(1)} & \\
       & \mathcal{B}_p^{(2)} \\
       &\\
       &\\
       & 
        \end{bmatrix}}_{\mathbf{B}_p}
       \begin{bmatrix}
       p_s^{(1)} \\
       p_s^{(2)}
       \end{bmatrix} 
        + \underbrace{\begin{bmatrix}
       0\\
       0\\
       B_q^{(3)} \\
       0\\
       0
        \end{bmatrix}}_{\mathbf{B}_q}q_d^{(3)}
        +\underbrace{\begin{bmatrix}
            \mathcal{G}_1(u^{(1)}, p_s^{(1)}) \\
            \mathcal{G}_2(u^{(2)}, p_s^{(2)}) \\
            \mathcal{G}_3(u^{(3)},\bar{e}_3u^{(1)} ) \\
            0 \\
            0
        \end{bmatrix}}_{\mathbf{g}(*)}.
       \end{split}
       \end{equation}
Here $u^{(i)}=[p^{(i)},q^{(i)}]$ and 
\begin{equation}\mathcal{G}_i (u^{(i)}, p_s^{(i)})
=
\begin{bmatrix}
 0 \\
 g(p^{(i)}_s,p^{(i)}, q^{(i)})
\end{bmatrix},
 \mbox{ and } \bar{\mathcal{B}}_p^3=\begin{bmatrix}
0, & 0, \ldots, &, 1, & 0, & \ldots, & \end{bmatrix}\otimes \mathcal{B}_p^3.
\end{equation} 
The mass flow nodal condition is represented by the 4th block row in~\eqref{eqn:exm_daes}, and the pressure nodal condition is given by the 5th block row and also the (3, 1) block of $\mathbf{K}$ in~\eqref{eqn:exm_daes}. The row vectors $e_1$, $e_2$, and $e_3$ are just elementary vectors with 1 or $-1$ on a certain position and zeros elsewhere, which select the corresponding variables for the nodal
conditions~\eqref{eqn:p_nodal}--\eqref{eqn:q_nodal}.

Note that the matrix $\mathbf{K}$ in~\eqref{eqn:exm_daes} is not uniquely defined. This is because we use $p_s^{(3)}=p_{n^{(1)}}^{(1)}$. We can also employ $p_s^{(3)}=p_{n^{(2)}}^{(2)}$, and this in turn moves $\bar{\mathcal{B}}_p^{(3)}$ from the (3, 1) block to the (3, 2) block of $\mathbf{K}$. 

Although we need extra variables for both, the pressure and mass flow, to assemble a global network model, we only introduce extra mass flow variables explicitly while the extra pressure variables can be obtained via applying some pressure nodal conditions directly. This reduces the redundancy in the network modeling. 

There is a degenerate case that a network has only one {long pipe}, i.e., this network has one supply node and one demand node, but no junction node. For such a degenerate network, which is equivalent to one pipe, we do not need to introduce extra variables since we already have the left and right boundary conditions. For a non-degenerate network, we have the following proposition.

\begin{proposition}\label{prop:alg}
Suppose that the network $\tilde{\mathcal{G}} = (\tilde{\mathcal{E}},\tilde{ \mathcal{N}})$ is a connected graph as in Remark \ref{rmk:graph}, and has $n_s$ supply pipes, $n_d$ demand pipes, and $n_j$ junction pipes. Here junction pipes are edges of the smoothed graph that are not supply pipes or demand pipes. Then the following relation between the number of extra variables $n_e$ and the number of extra algebraic constraints $n_a$ of the mathematical model holds:
    \[
       n_e= n_a = n_s + n_j.
    \]
\end{proposition}

\begin{proof}
As stated before, we only introduce extra variables for the mass flow of each supply and junction pipe, since the extra pressure variables are directly included at the process of network assembling. Then we have
    \[
        n_e = n_s + n_j,
        \] 
which is due to the fact that the mass flows at the demand pipes are already prescribed. 

The algebraic constraints are obtained via applying nodal conditions at the junction nodes. Suppose that the junction node $i$
has $n_{\text{in}}^{(i)}$ injection pipes, and $n_{\text{out}}^{(i)}$ outflow pipes, then we need $(n_{\text{in}}^{(i)}-1)$ equality constraints to apply the pressure nodal conditions for injection pipes since the pressure nodal conditions for outflow pipes are directly applied at the network assembling. We have one algebraic constraint to prescribe the mass flow nodal condition for junction node $i$. Therefore, we need $n_{\text{in}}^{(i)}$ algebraic constraints for junction node $i$. The sum 
over all the junction nodes of the network gives the overall number of algebraic constraints:
    \[
        n_a = \sum_i n_{\text{in}}^{(i)}.
    \]
On the other hand,
    \[
        \sum_i n_{\text{in}}^{(i)} = n_s + n_j,
        \]
as incoming pipes are never demand pipes and the sum over all incoming pipes is the number of all supply and all junction pipes.
 
\end{proof}

Proposition~\ref{prop:alg} states that the total number of extra variables is equal to the total number of algebraic constraints. This is very important for us to simulate the network model in the form of~\eqref{eqn:exm_daes}, which will be shown in the next section.

%% file: numerical.tex
With the help of our modeling we can represent the gas network as in~\eqref{eqn:exm_daes}. In general the number of algebraic equation will be much smaller than the number of differential equation in this description.  In this section, we will introduce fast numerical algorithms for the simulation of such a model. 

\subsection{Numerical Algorithms to Solve Differential Algebraic Equations}
Here, we reuse the notation from~\eqref{eqn:exm_daes} with simplifications. We obtain the general mathematical model,
\begin{equation}\label{eqn:model_daes}
    M\partial_t x = K x + Bu(t) + f(x, u(t)),
\end{equation}
where the mass matrix $M$ is singular when there is at least one junction node, and the right hand side function $f$ is nonlinear. In general, the mathematical model~\eqref{eqn:model_daes} is a large system of nonlinear DAE, where the size of the DAE~\eqref{eqn:model_daes} is proportional to the length of the overall pipes in the network. To solve/simulate such a DAE model is challenging and possibly slow. Related work either
focuses on exploiting the DAE structure such that the differential part and the algebraic part are decoupled, and one can solve these two parts
separately~\cite{GruHR16}, or reducing the so-called tractability index~\cite{GruHJetal14}. Here, we propose a fast numerical method by directly tracking the expensive numerical linear algebra.
To simulate the DAE model~\eqref{eqn:model_daes}, we discretize in time using the implicit Euler method, and at time step $k$, we have
\[
    M\frac{x^k - x^{k-1}}{\tau} = Kx^k + Bu^k + f(x^k, u^k),
\]
i.e., we need to solve the following nonlinear system of equations,
\begin{equation}\label{eqn:k_step}
    F(x)=(M-\tau K)x - \tau f(x, u^k) - M x^{k-1} - \tau Bu^k = 0,
\end{equation}
at each time step $k$ to compute the solution $x^k$. To solve such a nonlinear equation for the simulation of gas networks, some related work~\cite{GruJ15, Ascher1995} treats the nonlinear term $f(x, u^k)$ explicitly, i.e., $f(x^k, u^k)\approx f(x^{k-1}, u^k)$ since the input $u^k$ is known. This explicit approximation of the nonlinear term reduces the computational complexity, and yields a linear system. However, this approximation can lead to the necessity for smaller and smaller time steps. Here, we treat this nonlinear term implicitly and apply Newton's method to solve the nonlinear system~\eqref{eqn:k_step} to study the nonlinear dynamics of the network. Newton's method is described by Algorithm~\ref{alg:newton}.

\begin{algorithm}
    \caption{Newton's method to solve~\eqref{eqn:k_step}}\label{alg:newton}
            \begin{algorithmic}[1]
                \STATE {\textbf{Input:} maximal number of Newton steps $n_{\max}$, stop tolerance $\varepsilon_0$, initial guess $x_0$}
                \STATE $m=0$
                \WHILE {$m\leq n_{\max}\ \& \ \|F(x)\|\geq \varepsilon_0$}
                \STATE Compute the Jacobian matrix $D_F(x_m)=\frac{\partial}{\partial x}F|_{x=x_m}$ 
                \STATE Solve $F(x_m) + D_F(x_m)(x-x_m)=0$
                \STATE $m\gets m+1$, $x_m\gets x$
                \ENDWHILE
                \STATE {\textbf{Output:} solution $x\approx x_m$}
            \end{algorithmic}
\end{algorithm}

The biggest challenge for Algorithm~\ref{alg:newton} is to solve the linear system in line $5$ at each Newton iteration, since the Jacobian matrix $D_F(x_m)$ is large. Krylov subspace methods such as the generalized minimal residual (GMRES) method~\cite{Saad2003} or induced dimension reduction (IDR(s)) method~\cite{Sonneveld2008} are then appropriate to solve such a system. To accelerate the convergence of such a Krylov subspace method, we need to apply a preconditioning technique by exploiting the
structure of the Jacobian matrix $D_F(x)$.


\subsection{Matrix Structure}
The Jacobian matrix 
\begin{equation}\label{eqn:jacobian}
    D_F(x) = (M-\tau K) + \tau \frac{\partial}{\partial x}f(x, u),
\end{equation}
where the matrices $M$ and $K$ are two-by-two block matrices, and
\begin{equation}\label{eqn:m_and_k}
    M = \begin{bmatrix}
        \bar{M} & \\
                & 0
    \end{bmatrix},\quad K = \begin{bmatrix}
        K_{11} & K_{12} \\
        K_{21} & K_{22}
    \end{bmatrix}.
\end{equation}
Here $\bar{M}$ is block-diagonal, and the second block row of $A$ comes from the algebraic constraints of the networks by applying the nodal conditions introduced in Section~\ref{sec_net}. The size of $A_{11}$ is much bigger than the size of $A_{22}$ since $A_{11}$ comes from the discretization of the isothermal Euler equations over all the edges of the smoothed network, while the size of $A_{22}$ is equal to $n_s + n_j$ according to Proposition~\ref{prop:alg}. Moreover, the partial derivative of the nonlinear term $\frac{\partial}{\partial x}f(x, u)$ has the following structure
\begin{equation}\label{eqn:df}
    \frac{\partial}{\partial x}f = \begin{bmatrix}
        D_f & 0\\
        0 & 0
    \end{bmatrix},
\end{equation}
since the nonlinear term only acts on the differential part of the DAE~\eqref{eqn:exm_daes}. If the graph $\tilde{\mathcal{G}}$ is directed in a certain way and ordered a certain way, we can show that the first block of the Jacobian matrix is block lower triangular.

\begin{lemma}Given a graph with directed and undirected edges, such that the directed edges do not create cycles, we can always direct the undirected edges such that the resulting graph is a directed acyclic graph (DAG).
\end{lemma}
\begin{proof}
Take such a graph and remove all undirected edges. This results in a graph that may not be connected. However, it is a DAG so that we can order the nodes in a topological ordering~\cite{BanG08}. Applying such a topological ordering to the original graph, direct the undirected edges according to that topological ordering, which induces another DAG.
\end{proof}


This means, even by fixing the direction of the supply and demand pipes, we can redirect all the other edges in our graph $\tilde{\mathcal{G}}$ such that the resulting graph is a DAG. This is not unique, however induces unique directions in the original graph. Here, we use an example network to show how to build such a DAG.

\begin{figure}[H]
    \centering
    \subfigure[DAG with node ordering]
        {\label{fig:dag_1}
          \includegraphics[width=0.2\textwidth]{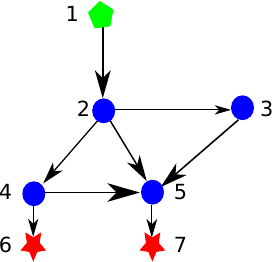}}
          \qquad  \qquad
    \subfigure[edges ordering and node ordering of DAG]
        {\label{fig:dag_2}
          \includegraphics[width=0.5\textwidth]{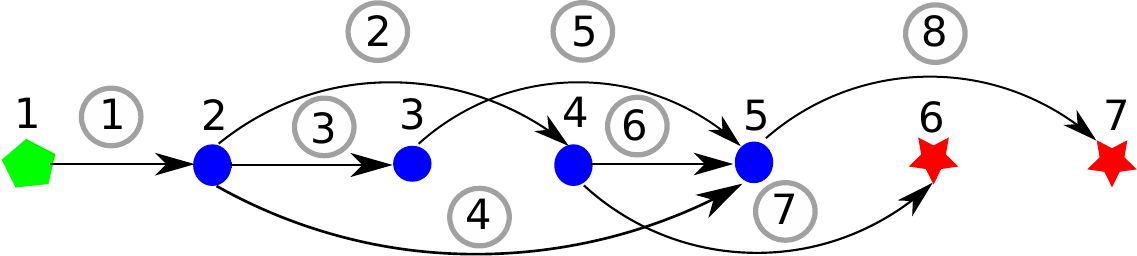}}
    \caption{An illustrative network example of a DAG}\label{fig:dag}
\end{figure}

The network in Figure~\ref{fig:dag_1} is a smoothed graph that represents a gas network. Only the direction of the supply and demand pipes are fixed. We start ordering the nodes of the graph from the supply nodes, and end up with the demand nodes, which gives a topological ordering of the nodes. Now, we can plot the graph along a line as in Figure~\ref{fig:dag_2}. The directions of the undirected edges of the graph are picked up by pointing away from lower order nodes to higher order nodes, which induces a DAG. 

\begin{lemma} Given a DAG, we can order the edges in such a way that at every node all incoming edges have a lower order as all the outgoing edges, or it doesn't have an incoming edge. We call this ordering \textbf{direction following (DF) ordering} \end{lemma}

\begin{proof}
    Since we have a DAG, there exist a topological ordering of the nodes. This means if an edge goes from node $a$ to node $b$, then $b$ has a higher order than $a$. We now order the edges, by their starting node. This means an edge has a higher order if its starting node has a higher order. If two edges have the same starting node, then it doesn't matter in which order we list them, we just pick one. Once we have this ordering of the edges based on the order of the nodes we ensure that all outgoing edges at a node have a higher order as all the incoming edge of a node. 
\end{proof}

Note that the DF ordering is not unique. A DF ordering example of the network in Figure~\ref{fig:dag_1} is given by Figure~\ref{fig:dag_2}, where the index above each edge is the order of such an edge.

\begin{definition}(Smoothed direction following (SDF) gas network graph) This is a directed acyclic graph, whose boundary nodes are supply or demand nodes, directions are away from supply nodes and towards demand nodes. All edges are sorted with the DF ordering and there are no nodes in the graph that connect exactly two edges. \label{def:sdf}
\end{definition}

From now on, we assume that our modeling is such that $\tilde{\mathcal{G}}$ is a SDF gas network graph. Then we have the following proposition to illustrate the structure of the partial derivative of the nonlinear term~\eqref{eqn:df}.
\begin{proposition}\label{prop:low}
Given a SDF gas network graph we can construct the DAE system in such a way that $D_f$ in~\eqref{eqn:df} has a block lower-triangular structure. 
\end{proposition}

\begin{proof}
    The block of $f$ corresponding to the $i$-th pipe has the structure, 
    \[
       f_i=\mathcal{G}_i = \begin{bmatrix}
       0 \\
       g(p_{\text{in}},p^{(i)},q^{(i)}) 
       \end{bmatrix},
    \]
where the structure of $g(p_{\text{in}},p^{(i)},q^{(i)})$ is given by~\eqref{eq:gfunction}. Here
    \[
        p_{\text{in}}=\begin{cases}
            p_s, & \text{when $i$-th pipe is a supply pipe}, \\
            p_{\text{out}}^{(j)}, & \text{where pipe $j$ is a selected incoming pipe into the node where pipe $i$ is outgoing}.
        \end{cases}
    \]
    By the selected ordering, we always have $j<i$. Then the upper triangular blocks of $D_f$ are  0. The diagonal blocks are given by
\[
  \begin{bmatrix}
   0 & 0 &0\\
      0& \frac{g(p_{\text{in}},p^{(i)},q^{(i)})}{\partial p^{(i)}}  &    \frac{g(p_{\text{in}},p^{(i)},q^{(i)})}{\partial q^{(i)}} 
  \end{bmatrix},
   \]
   which is again block lower triangular and in particular, with an easy structure of the diagonal blocks, since $\frac{g(p_{\text{in}},p^{(i)},q^{(i)})}{\partial q^{(i)}}$ in our discretization is tridiagonal.
%
%
%
%
\end{proof}

Similar to Proposition~\ref{prop:low}, we can also show that the $(1, 1)$ block of $K$ in~\eqref{eqn:m_and_k} has a lower-triangular block structure.

\begin{proposition}\label{prop:a11}
  Given an SDF gas network graph, $K_{11}$ in~\eqref{eqn:m_and_k} is also block lower-triangular structure.
\end{proposition}

\begin{proof}
    For the $i$-th block row of $K_{11}$, the off-diagonal blocks are zero if the $i$-th pipe is a supply pipe. If the $i$-th pipe is not a supply pipe and connected with other pipes, then the off-diagonal block $(i,\ j)$ is nonzero if the $j$-th pipe corresponds to one of the flow injection pipes of the $i$-th pipe. This is because the pressure nodal condition~\eqref{eqn:p_nodal} is applied for the $i$-th pipe. According to Definition~\ref{def:sdf}, we can pick the pressure condition for the $i$-th pipe by any injecting pipe $j$, which are all of lower order and therefore ensure $i>j$, and this completes the proof.
\end{proof}

If we partition the Jacobian matrix~\eqref{eqn:jacobian} by a 2-by-2 block structure as in~\eqref{eqn:m_and_k}, then we have the following theorem to illustrate the structure of the $(1, 1)$ block of the Jacobian matrix. 
\begin{theorem}\label{thm:jacobian}
    Given an SDF gas network graph we are able to model the system in such a way, that the $(1, 1)$ block of the Jacobian matrix~\eqref{eqn:jacobian} has a block lower-triangular structure. 
\end{theorem}

\begin{proof}
    According to~\eqref{eqn:jacobian}--\eqref{eqn:df}, the $(1, 1)$ block of the Jacobian matrix $D_F(x)$ is,
    \[
        \bar{M}-\tau K_{11} + \tau D_f.
        \]
    According to Proposition~\ref{prop:low} and Proposition~\ref{prop:a11}, $K_{11}$ and $D_f$ are block lower-triangular matrices. 
    Since $\bar{M}$ is a block diagonal matrix, this completes the proof. 
\end{proof}

Next, we use a benchmark network from~\cite{Roggendorf15}, shown in Figure \ref{fig:big_net} to show the structure of the Jacobian matrix $D_F(x)$ of the first Newton iteration for the first time step, i.e., $D_F(x_{1}^1)$ before and after applying the DF ordering. The network parameters are also given in~\cite{Roggendorf15}. We set the mesh size for the FVM discretization to be $20$ meters, i.e., $h=20$. The sparsity pattern of $D_F(x_1^1)$
before and after the DF ordering are given by Figure~\ref{fig:spy_plot}.

\begin{figure}[H]
    \centering
    \includegraphics[width=0.6\textwidth]{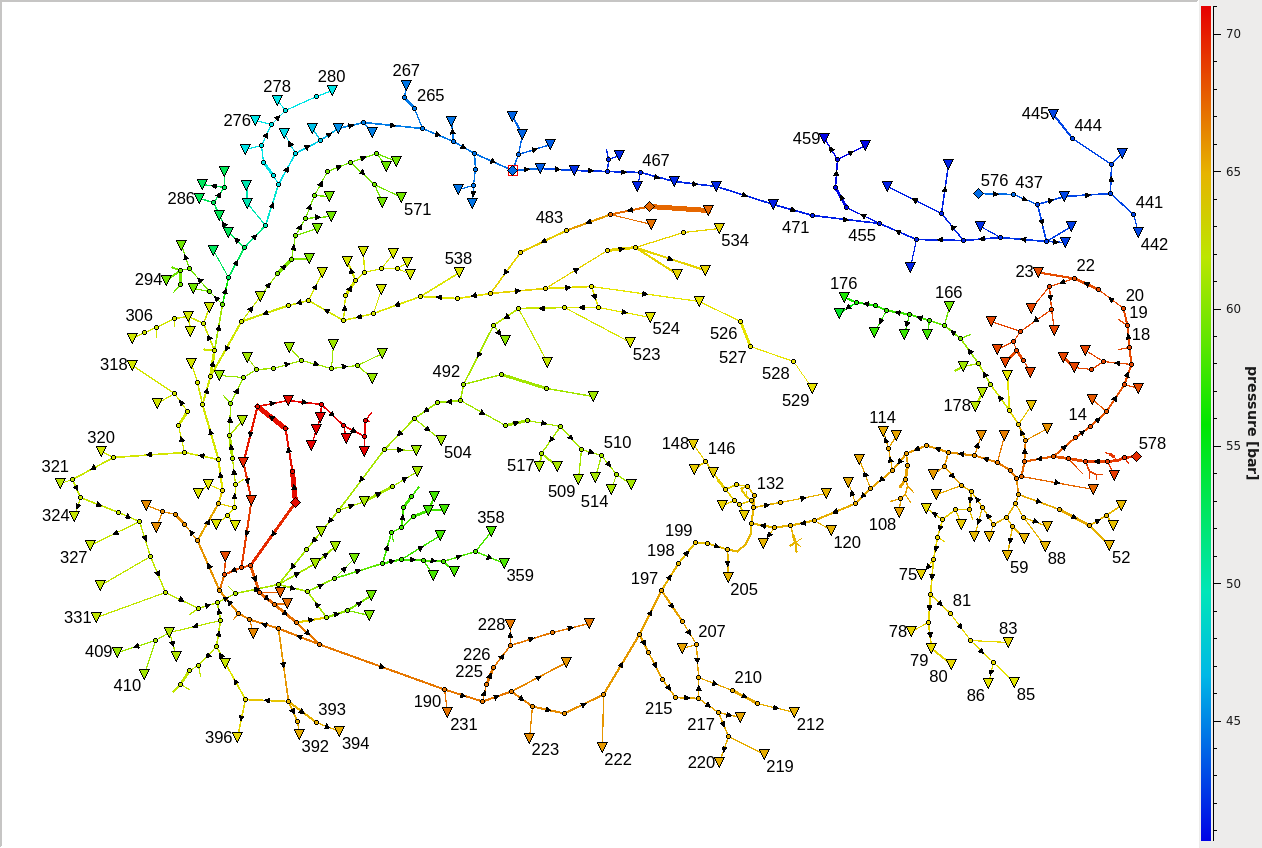}                                                                  
    \vspace{-0.2cm}
    \caption{Big benchmark network in~\cite{Roggendorf15}}\label{fig:big_net}
\end{figure}
\vspace{-0.6cm}

\begin{figure}[H]
    \centering
    \subfigure[without DF ordering]
    {\label{fig:without}
        \includegraphics[width=0.38\textwidth]{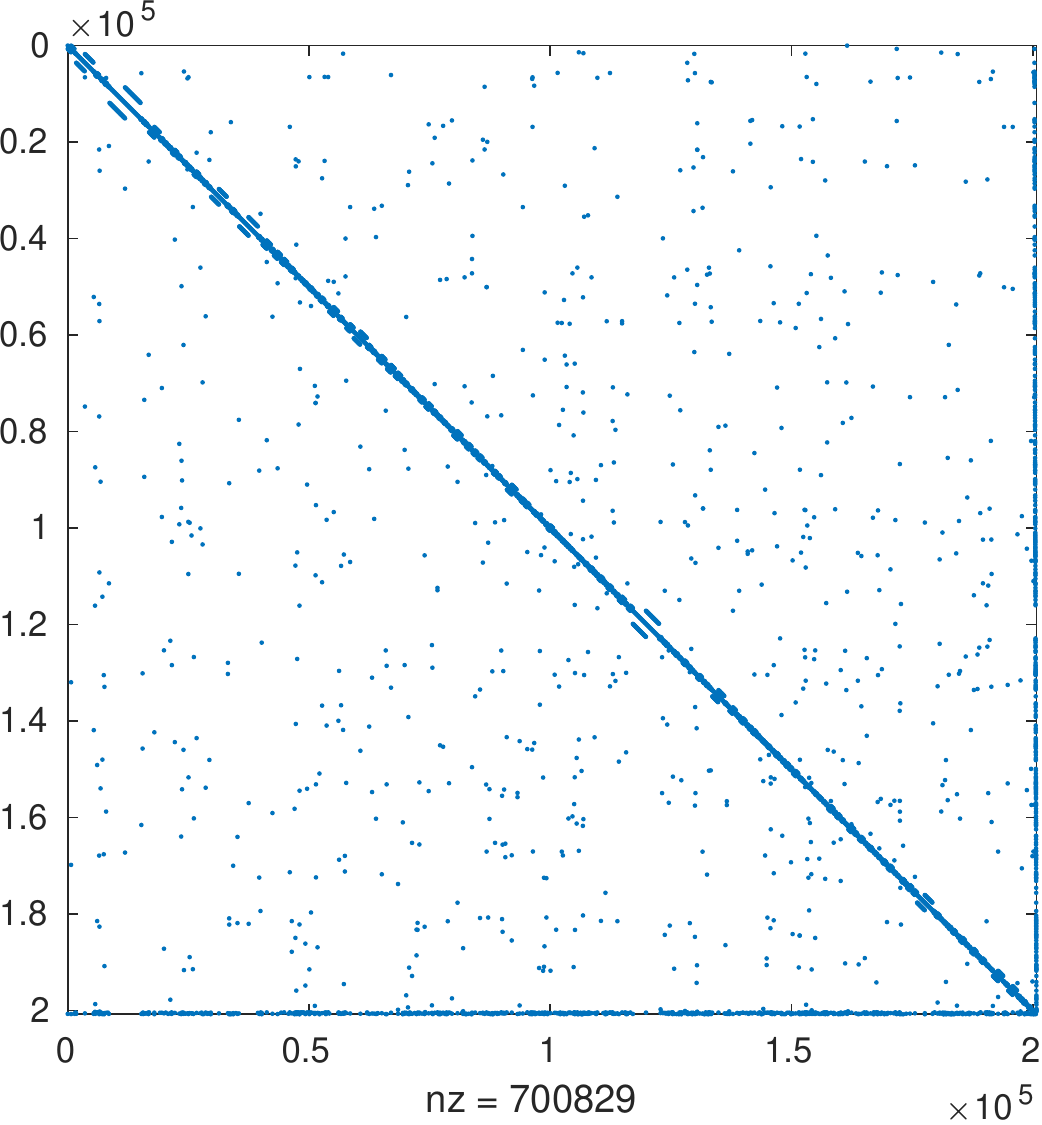}}\qquad\qquad
    \subfigure[with DF ordering]
    {\label{fig:with}
\includegraphics[width=0.38\textwidth]{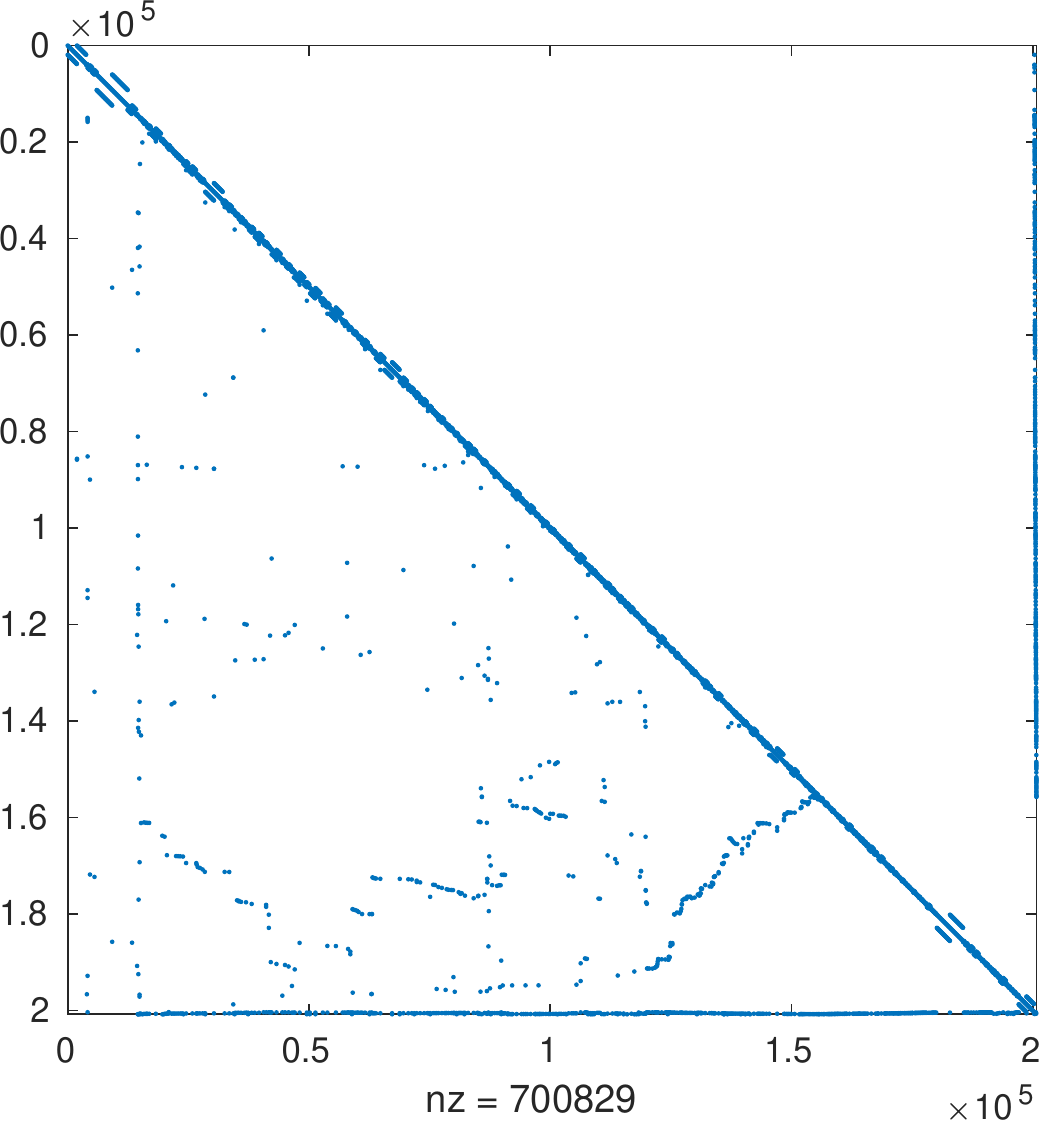}}                                                                  
    \vspace{-0.2cm}
    \caption{Sparsity pattern of $J$ without and with DF ordering}\label{fig:spy_plot}
\end{figure}
\vspace{-0.2cm}
Figure~\ref{fig:spy_plot} shows that the Jacobian matrix $D_F(x)$ is a sparse matrix. After applying the DF ordering, the $(1, 1)$ block has a block lower-triangular structure, and the size of the $(1, 1)$ block is much bigger than the $(2, 2)$ block. For the case when the mesh size is $20$ meters, the $(1, 1)$ block is a $200,348\times 200,348$ block lower-triangular matrix while the size of the $(2, 2)$ block is $417\times 417$. In general, the size
of the $(2, 2)$ block of the Jacobian matrix is fixed since it equals the number of algebraic constraints. According to Proposition~\ref{prop:alg}, it equals $n_s+n_j$, where $n_s$ is the number of supply pipes and $n_j$ is the number of junction pipes. The size of the $(1, 1)$ block depends on the mesh size and equals twice the total pipeline length divided by the mesh size. Therefore, it is much
bigger than $n_s+n_j$, if the total length of the pipelines is much larger compared to the mesh size.
\subsection{Preconditioning Technique}

The specific structure of the Jacobian matrix can be exploited to solve the Jacobian system fast for the simulation of the gas network. Recall that to simulate the discretized gas network model, we need to apply Algorithm~\ref{alg:newton}, while we need to solve a Jacobian system at each Newton iteration for each time step $k\ (k=1,\ 2,\ \cdots, n_t)$. To solve the Jacobian system, we exploit the 2-by-2 structure of the Jacobian matrix. Here we write the Jacobian
matrix $D_F(x)$ as
\[
    D_F(x) = \begin{bmatrix}
        D_{F_{11}} & D_{F_{12}} \\
        D_{F_{21}} & D_{F_{22}} 
    \end{bmatrix}.
\]
Note that the Jacobian matrix $D_F(x)$ has a special structure, which is called \textit{generalized saddle-point structure}~\cite{Benzi2005}. This enables us to make use of the preconditioning techniques designed for the generalized saddle-point systems to solve the Jacobian system. Generalized saddle-point systems come from many applications, such as computational fluid dynamics~\cite{elman2005finite}, PDE-constrained optimization~\cite{Rees2010}, optimal flow control~\cite{Qiu2015_phd}. Many efforts
have been dedicated to the efficient numerical solution of such systems using preconditioning techniques~\cite{Pearson2015, Porcelli2015, Wathen2017, Pestana2015, Stoll2015}, we recommend~\cite{Benzi2005, Wathen2015} for a general survey of preconditioning generalized saddle-point systems. 

We can compute a block LU factorization by
\begin{equation}\label{eqn:block_lu}
    D_F(x) = \begin{bmatrix}
        D_{F_{11}} & \\
        D_{F_{21}} & S
    \end{bmatrix}
    \begin{bmatrix}
        I & D_{F_{11}}^{-1}D_{F_{12}} \\
          & I
    \end{bmatrix}.
\end{equation}
Here $S=D_{F_{22}}-D_{F_{21}}D_{F_{11}}^{-1}D_{F_{12}}$ is the Schur complement of $D_{F}(x)$. According to Theorem~\ref{thm:jacobian}, $D_{F_{11}}$ has a block lower-triangular structure, and the size of $D_{F_{22}}$ is much smaller than $D_{F_{11}}$. Therefore, we can compute the Schur complement $S$ by block forward substitution, and apply the following preconditioner,
\begin{equation}\label{eqn:preconditioner}
    P = \begin{bmatrix}
        D_{F_{11}} & \\
        D_{F_{21}} & S
    \end{bmatrix},
\end{equation}
to solve the Jacobian system using a preconditioned Krylov solver. Associated with the block LU factorization~\eqref{eqn:block_lu}, we can immediately see that the preconditioned spectrum $\lambda(P^{-1}D_{F}(x))=\{1\}$. Moreover, the minimal polynomial of the preconditioned matrix $P^{-1}D_{F}(x)$ has degree 2, so that a method like generalized minimum residual (GMRES)~\cite{Saad2003} would converge in at most two steps~\cite{Benzi2005}. 

At each iteration of the Krylov solver, we need to solve the system 
\[
    \begin{bmatrix}
        D_{F_{11}} & \\
        D_{F_{21}} & S 
    \end{bmatrix}
    \begin{bmatrix}
        y_1\\
        y_2
    \end{bmatrix}
        =\begin{bmatrix}
            r_1 \\
            r_2
        \end{bmatrix},
        \]
which can be solved easily since $D_{F_{11}}$ is a block lower-triangular system, and $S=D_{F_{22}}-D_{F_{21}}D_{F_{11}}^{-1}D_{F_{12}}$ can be computed directly since the size of $S$ is much smaller than $D_{F_{11}}$. Note that at each time step $k$, we need to solve a nonlinear system using Newton's method, and we need to apply a preconditioned Krylov subspace method to solve a Jacobian system at each Newton iteration. For such a preconditioned Krylov solver, we need to compute the
Schur complement $S$ at each Newton iteration. This can still be computationally expensive for gas network simulation within a certain time horizon. We can further simplify the preconditioner by applying a fixed preconditioner $P_1$ for all Newton iterations and all time steps, i.e., we choose
\begin{equation}\label{eqn:p0}
    P_1 = \begin{bmatrix}
        D_{F_{11}}^1 & \\
        D_{F_{21}}^1 & S^1
    \end{bmatrix},
    \end{equation}
where $P_1$ comes from the block LU factorization of the Jacobian matrix $D_F^1(x_1)$ of the first Newton iteration for the first time step, and $S^1=D_{F_{22}}^1-D_{F_{21}}^1(D_{F_{11}}^1)^{-1}D_{F_{12}}^1$. Note that for the preconditioner $P_1$, we just need to compute the Schur complement once and  use it for all the Newton iterations of all time steps.

{Next, we show the performance of the DF ordering for the Schur complement $S^1$ computation. Again, we use the network given in Figure~\ref{fig:big_net} as an example and perform a FVM discretization of the network using different mesh sizes. We report the computational results in Table~\ref{tab:time_s}, where all timings are given in seconds. Here, $h$ is the mesh size, and $\# D_F$ represents the size of the Jacobian matrix $D_F$ given by~\eqref{eqn:block_lu}. The computations of $(D_{F_{11}}^1)^{-1}$ are performed using the MATLAB {\tt backslash} operator for both cases.}

{We have noticed that for medium problem sizes, the advantage of computing $S^1$ using the block lower-triangular structure obtained from the DF ordering over that without using the DF ordering is not very obvious. This is due to the fact that while computing $S^1$ using the block lower-triangular structure of $D_{F_{11}}^1$ obtained from the DF ordering, MATLAB has an overhead calling the block forward substitution. This overhead is comparable with the hardcore computation time for
medium problem sizes. When the problem size gets bigger, this overhead is less comparable with the hardcore computation time, which is demonstrated by the results in Table~\ref{tab:time_s}. We believe that the advantage of computations with the DF ordering over computations without the DF ordering will become more apparent once we use a tailored high performance computation implementation.}

\vspace{-0.4cm}

\begin{table}[H]
    \centering
    \caption{\footnotesize Computational time (seconds) for Schur complement $S^1$}
    \label{tab:time_s}
    \vspace{0.3cm}
    \begin{tabular}{cccc}
        \toprule
        $h$  & $\# D_F$ & with DF  & without DF   \\
        \midrule
        20  & 2,01e+05 & $8,12$      & 8,75     \\ 
        \addlinespace
        10  & 3,97e+05 & $17,84$     & 19,14      \\ 
        \addlinespace
        5   & 7,91e+05 & $38,44$     & 41,75    \\ 
        \addlinespace
        2.5 & 1,58e+06 &$81,42$     & 87,77    \\ 
        \bottomrule
    \end{tabular}
\end{table}

\vspace{-0.1cm}

Since $P_1$ is a good preconditioner for $D_{F}^1(x_1)$, it is also a good preconditioner to the Jacobian matrix $D_F(x)$ at the other Newton iterations and other time steps, if  close to $D_{F}^1(x_1)$. This is true for the gas networks since the Jacobian matrix~\eqref{eqn:jacobian} has two parts, i.e., the linear part and the linearized part. The linear part is dominant since it models the
transportation phenomenon of the gas while the nonlinear term acts as the friction term for such a transportation. This makes $P_1$ a good preconditioner for solving the Jacobian systems for all Newton steps of all time steps, as it will be demonstrated by numerical experiments in the next section. Note that if we keep updating the preconditioner~\eqref{eqn:preconditioner} more often than simply using a single preconditioner $P_1$ in~\eqref{eqn:p0}, we will obtain better performance for the
preconditioned Krylov solver, which in turn needs more time for preconditioner computation. A compromise has to be made to achieve the optimal performance for the gas network simulation in the term of total computational time. 

\vspace{-0.1cm}
\begin{figure}[H]
    \centering
    \includegraphics[width=0.5\textwidth]{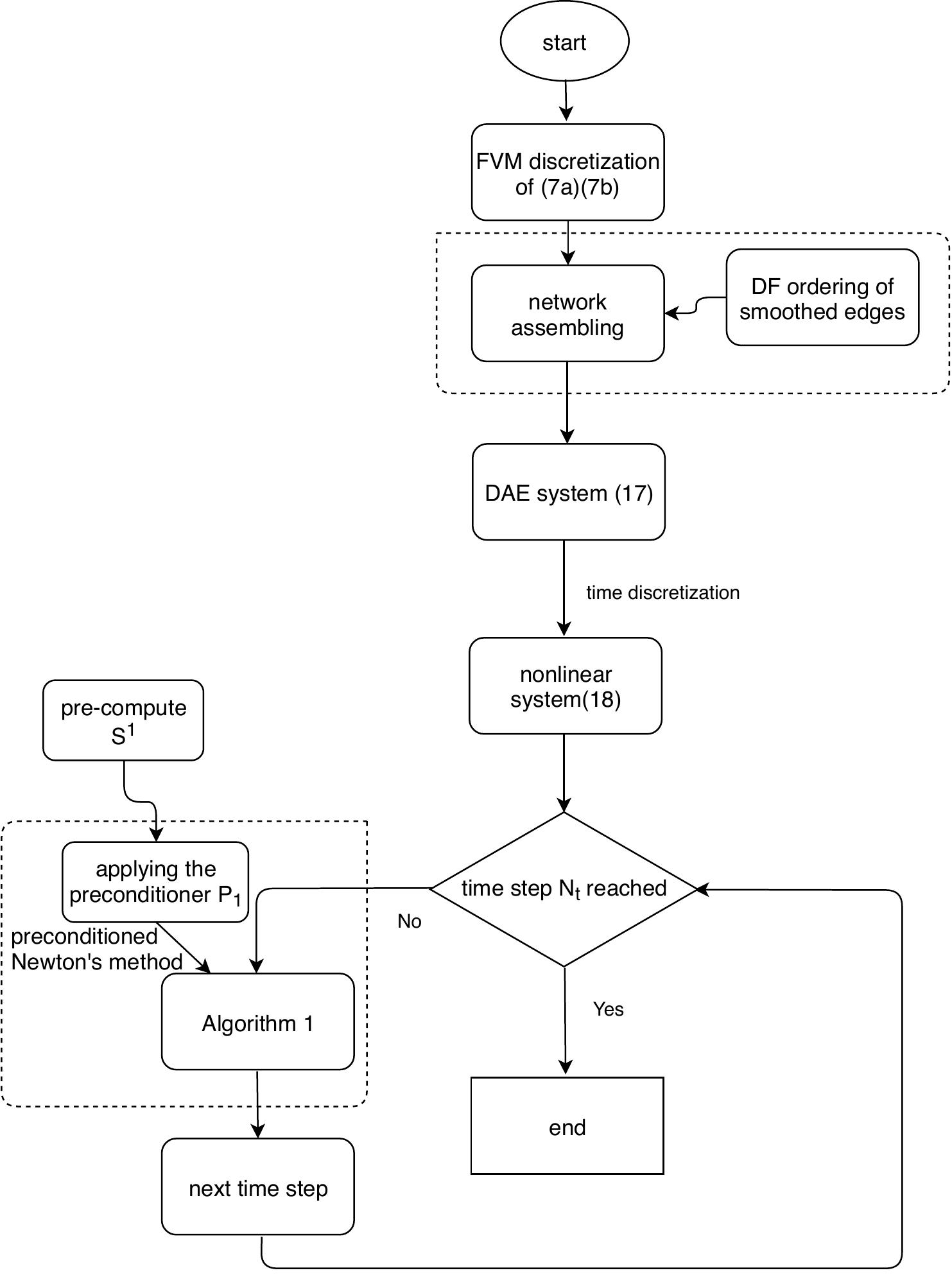}
    \vspace{-0.3cm}
    \caption{Computational diagram for gas network simulation}\label{fig:diagram}
\end{figure}
\vspace{-0.1cm}

By applying the preconditioner $P_1$ in~\eqref{eqn:p0}, we show the computational diagram to illustrate the process of gas network simulation in Figure~\ref{fig:diagram}.

%% file: results.tex
In this section, we report the performance of our numerical algorithms for the simulation of the gas networks. We apply our numerical algorithms to the benchmark problems of several gas networks given in~\cite{GruHKetal13, GruHJetal14, GruJ15, GruHR16} to show the performance of our methods. All numerical experiments are performed in MATLAB 2017a on a desktop with Intel(R) Core(TM)2 Quad CPU Q8400 of 2.66GHz, 8 GB memory and the Linux 4.9.0-6-amd64 kernel.

\subsection{Comparison of Discretization Methods}
In this section, we compare the performance of the finite volume method (FVM) with that of the finite difference method (FDM) for the discretization of the gas networks. We apply both the FVM and FDM to a pipeline network illustrated in Figure~\ref{fig:single_net}. Parameter settings for this pipeline network are given in~\cite{GruHR16}.

\begin{figure}[H]
         \centering
         \includegraphics[width=0.3\textwidth]{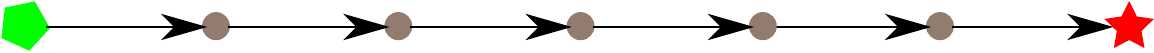}
         \caption{Pipeline network in~\cite{GruHJetal14}}\label{fig:single_net}
\end{figure}

We discretize the pipeline network using FVM and FDM with different mesh sizes, and the discretized pipeline networks result in ordinary differential equations (ODEs) since there is no algebraic constraint. We simulate the ODE systems using the routine {\tt ode15s} in MATLAB over the time horizon $[0,\ 10^5]$ with the same setting of the initial condition for the ODEs. The computational results are given in
Figure~\ref{fig:results_single}, where the $x$-axis represents the mesh sizes in meters. 

\begin{figure}[H]
   \centering
   \subfigure[time steps]
   {\label{fig:time_steps}
   \includegraphics[width=0.45\textwidth]{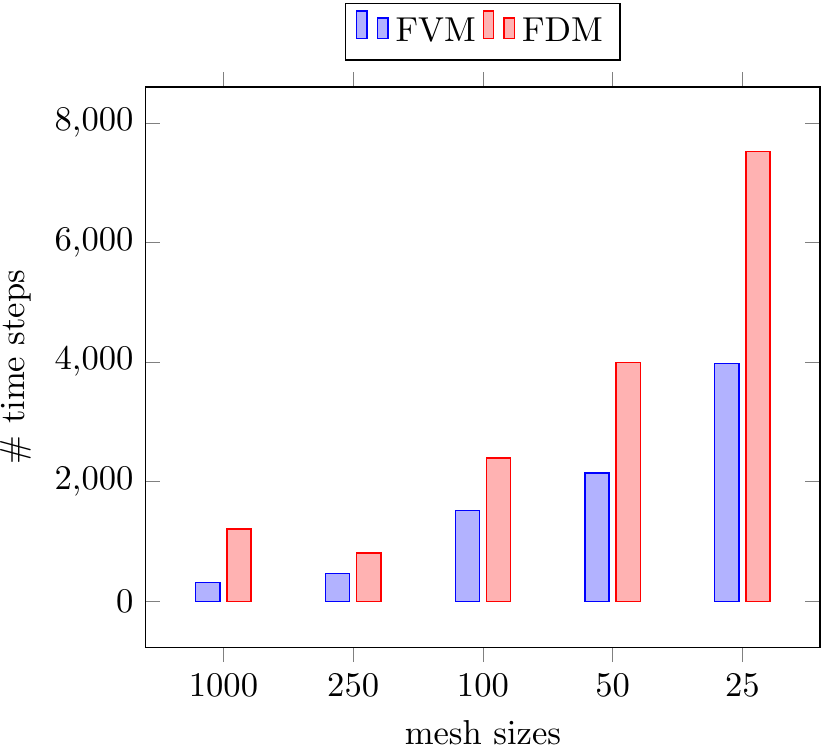}}
   \qquad
   \subfigure[total timing]
   {\label{fig:timing}
   \includegraphics[width=0.45\textwidth]{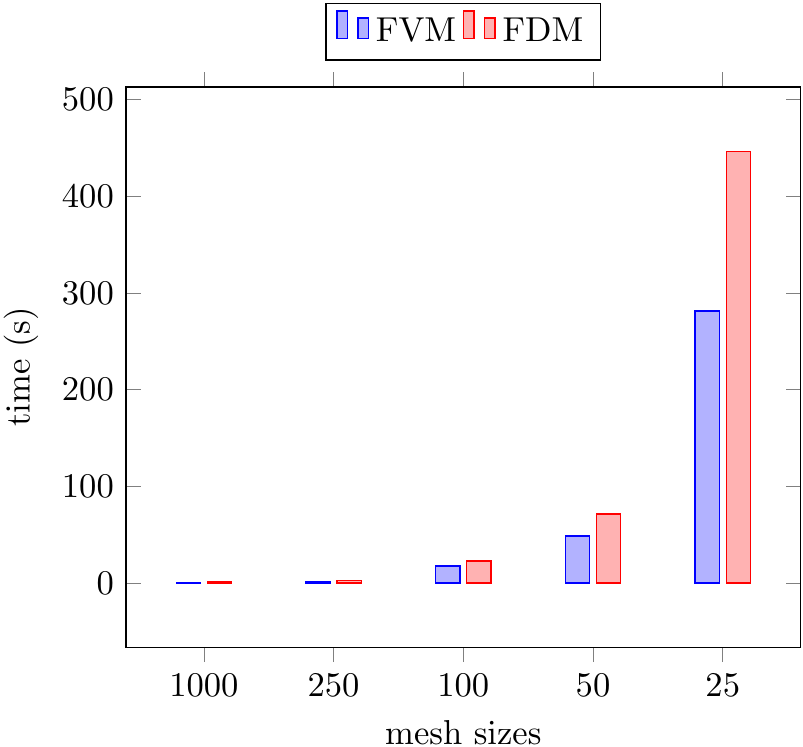}}
     \vspace{-0.3cm}
\caption{Comparison of FVM and FDM for a single pipe network}\label{fig:results_single}
\end{figure}

Figure~\ref{fig:time_steps} shows the number of time steps that {\tt ode15s} uses to simulate the ODEs given by FVM and FDM discretization over the time horizon $[0,\ 10^5]$. One can see that for a given mesh size, we need less time steps for the ODE given by the FVM discretization than the ODE given by the FDM discretization which results in less total computation time for the simulation of the ODE given by the FVM discretization , which is also shown in Figure~\ref{fig:timing}. The background mechanism is not clear since {\tt ode15s} behaves like a black-box. 

Next, we use another network to show that with the same mesh, the model given by the FVM discretization gives more accurate results than the FDM discretization. The test network is given in Figure~\ref{fig:medium_network}, where the network parameters are given in~\cite{GruHJetal14}. We use the FVM and FDM methods to discretize the network in Figure~\ref{fig:medium_network} and apply the computational method depicted in Figure~\ref{fig:diagram}. We choose different mesh sizes for the discretization, and fix the step size for the time discretization to be one second, i.e., $\tau=1$. We plot the mass flow at the supply node $57$ in Figure~\ref{fig:results_medium}. 

The mass flow at node 57 computed by using different discretized DAE models~\eqref{eqn:k_step} is plotted in Figure~\ref{fig:qs_57}, showing similar dynamical behavior of the different models. However, when  we look at the dynamics of the mass flow at the first 5 seconds, we can see quite a big difference in Figure~\ref{fig:qs_57_zoom}. With the mesh refinement, the solutions of the model given by both the FVM and the FDM discretizations converge. Moreover, we can infer that we can use a bigger mesh size for the FVM discretization than for the FDM discretization to get the same accuracy.

\begin{figure}[H]
       \centering
       \includegraphics[width=0.55\textwidth]{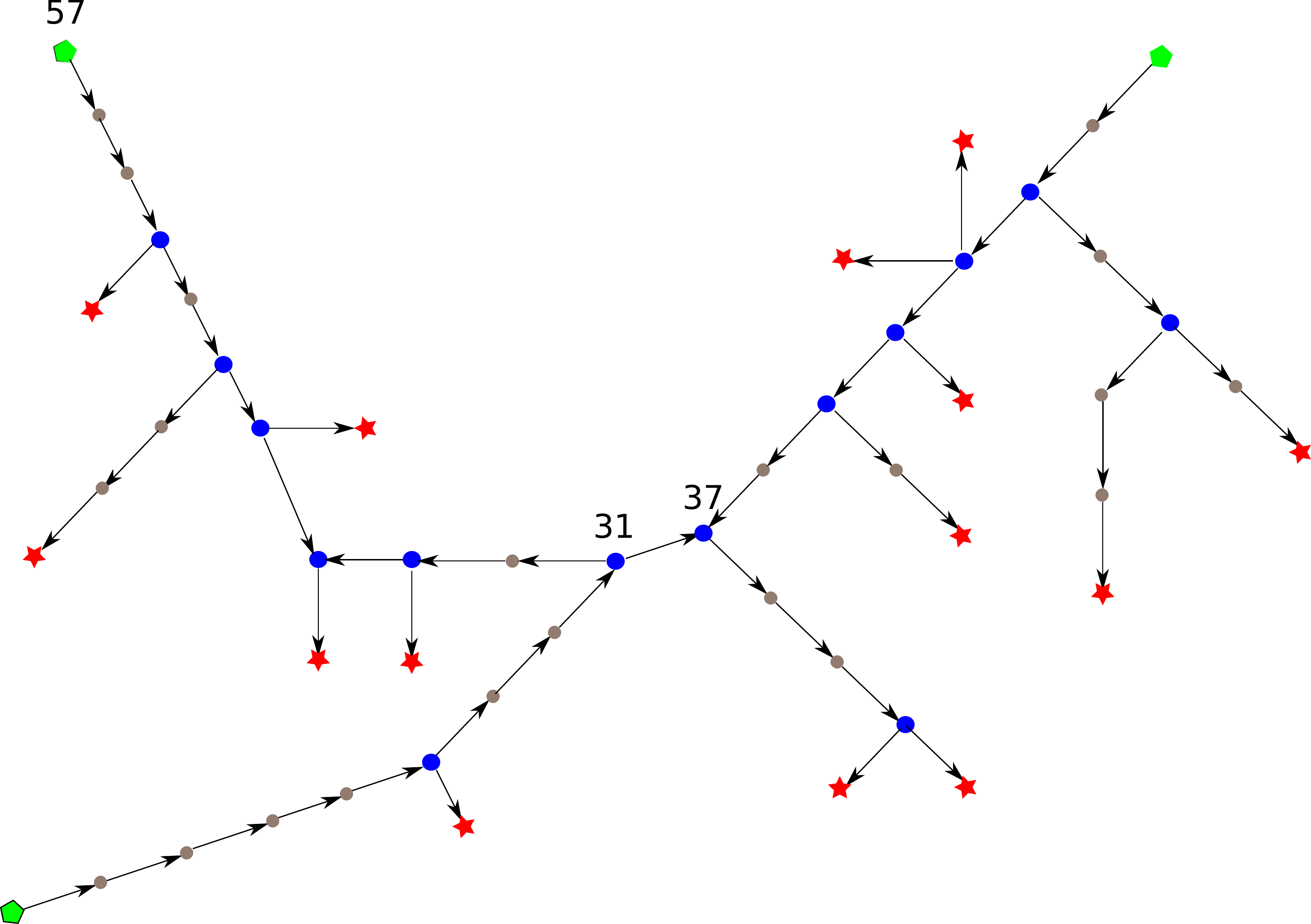}
       \caption{Medium size network}\label{fig:medium_network}
\end{figure}

The computational results given by Figure~\ref{fig:results_single} and Figure~\ref{fig:results_medium} show that when we use the same mesh size to discretize the network, the model given by the FVM discretization is more accurate than the model by the FDM discretization. To get the same model accuracy, we can use a bigger mesh size to discretize the network by FVM than that by FDM\@. This in turn yields a smaller model given by the FVM discretization than the model given by the FDM
discretization. This in turn means that the FVM discretized model is easier to solve than the
FDM discretized model.

\begin{figure}[H]
  \centering
  \subfigure[supply flow at node $57$]
  {\label{fig:qs_57}
  \includegraphics[width=0.47\textwidth]{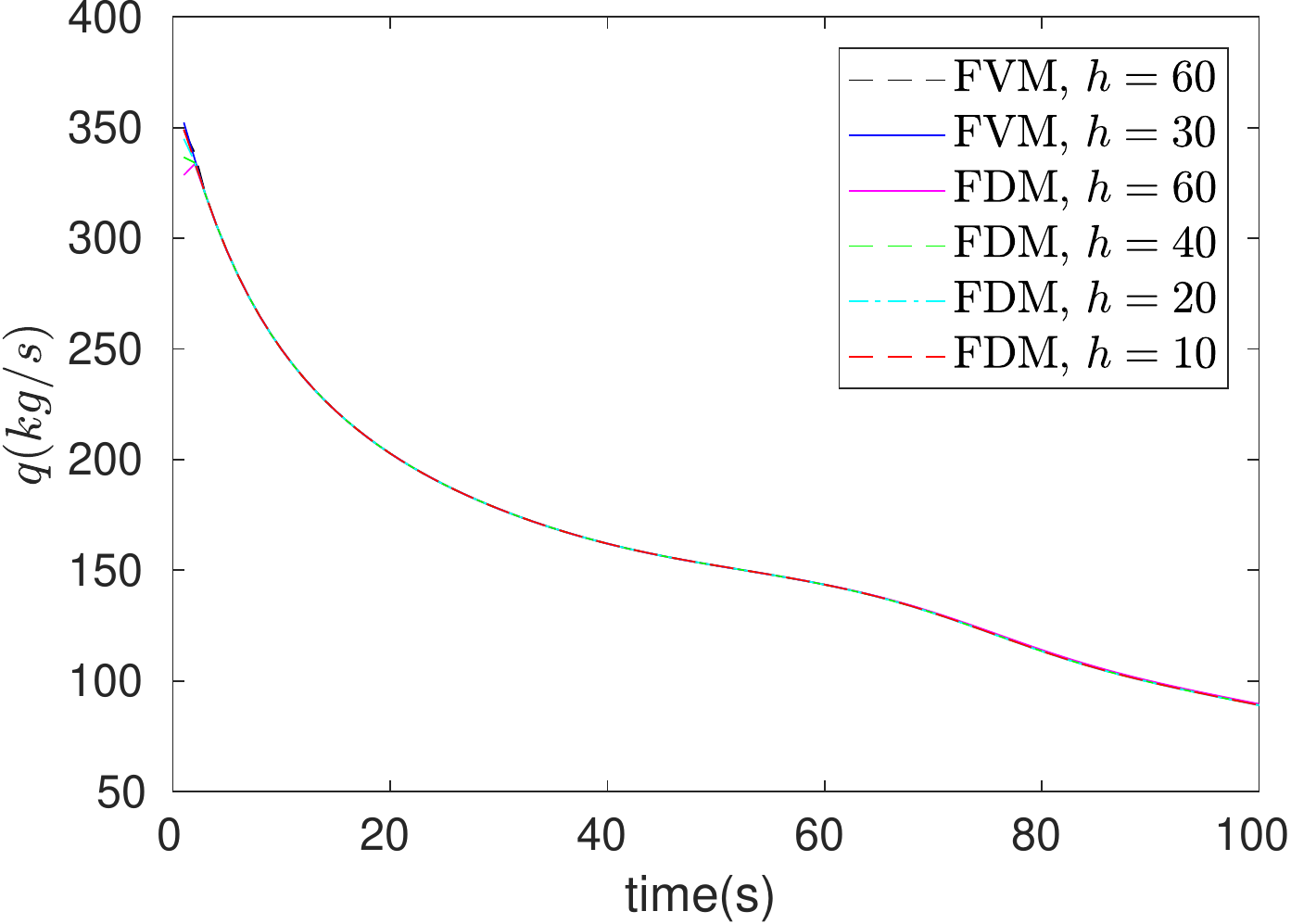}}
  \qquad
  \subfigure[zoomed supply flow at node $57$]
  {\label{fig:qs_57_zoom}
  \includegraphics[width=0.46\textwidth]{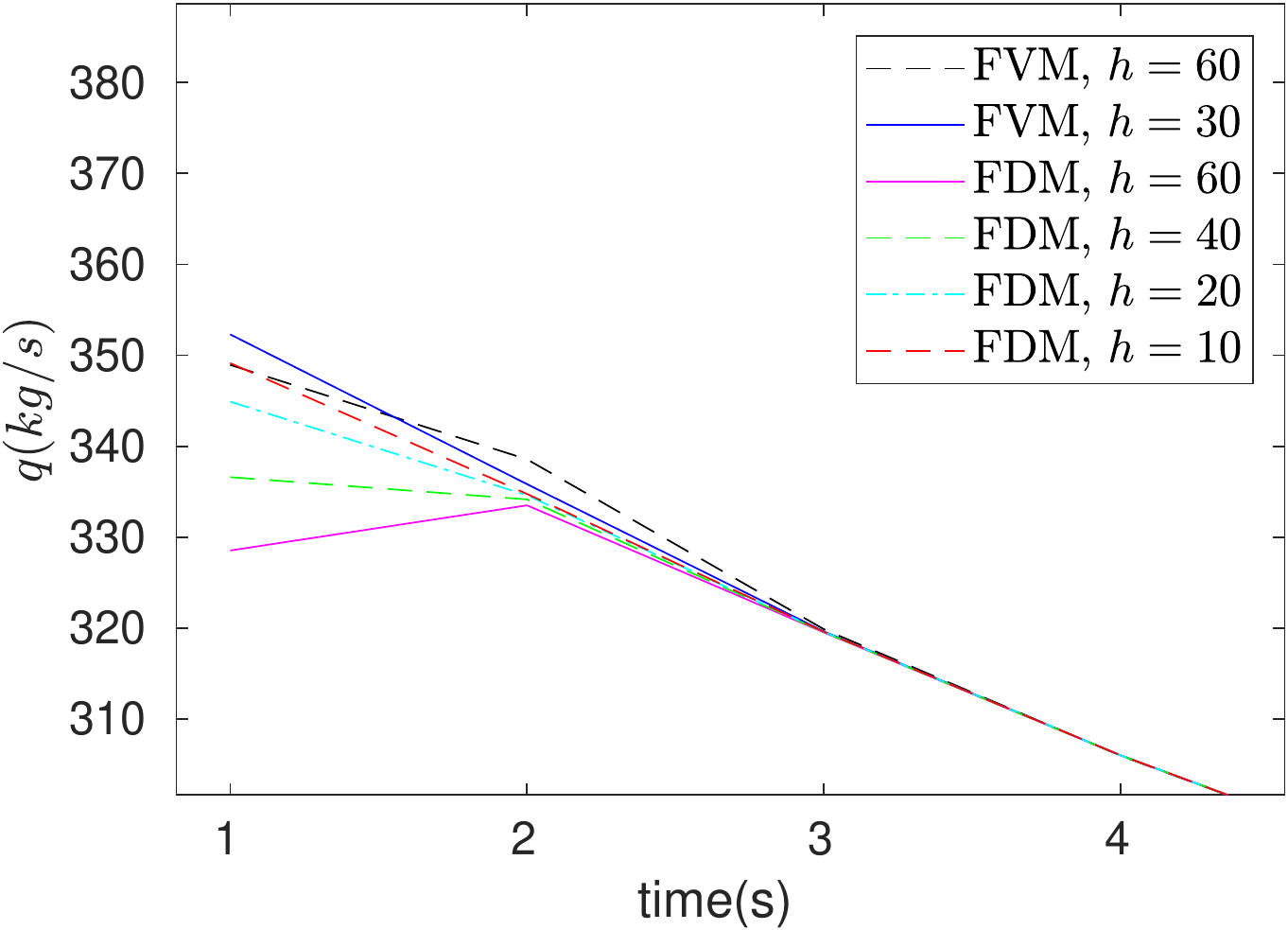}}
    \vspace{-0.3cm}
\caption{Comparison of FVM and FDM for a medium network}\label{fig:results_medium}
\end{figure}

We also plot the condition number of the Jacobian matrix ($\kappa(D_F)$) of all the Newton iterations for the first time step with a mesh size $h=60$ to discretize the DAE, which are given in Figure~\ref{fig:condition_number}. It illustrates that the condition number of the Jacobian matrix of the FVM discretized model is about 10 times smaller than the condition number of the Jacobian matrix of the FDM discretized model, which makes solving such a FVM discretized model easier than solving
a FDM discretized model.

Computational results in Figure~\ref{fig:results_single}--\ref{fig:condition_number} show that the finite volume method has a big advantage over the finite difference method. When using the same mesh size for discretization, FVM gives a more accurate model than the FDM discretization. Moreover, the Jacobian matrix from the FVM discretized model has a better condition number than the Jacobian matrix from the 
FDM discretized model, which makes it easier to simulate the FVM discretized model. To get the same model accuracy, the size of the FVM discretized model is smaller than the size of the FDM discretized model, and it is therefore computationally cheap. For the comparison of the finite element method (FEM) with FDM for the gas network simulation, we refer to an early study in~\cite{Osiadacz1989}, where the authors preferred FDM due to the comparable accuracy with FEM and less
computational time.

\begin{figure}[H]
   \centering
   \includegraphics[width=0.45\textwidth]{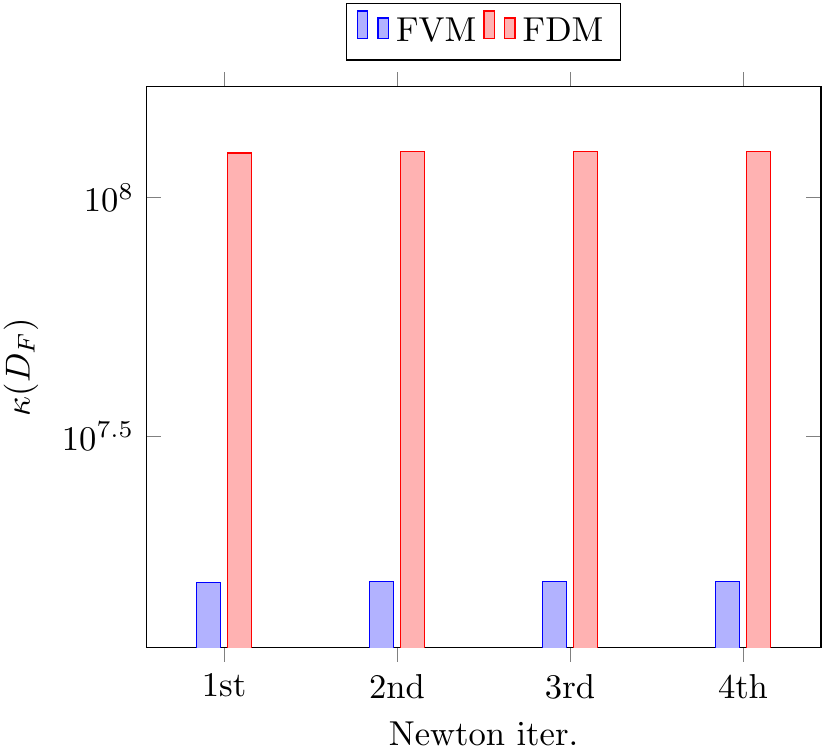}
     \vspace{-0.3cm}
 \caption{Condition number of the Jacobian matrix, 1st time step, $h=60$}\label{fig:condition_number}
\end{figure}

\subsection{Change of flow direction}
The basis of our modeling is a directed graph, with the implicit assumption that the gas flow follows that direction. However, the flow direction may change due to the change of operation conditions of the gas network. In this part, we show that the direction is just a theoretical construction but that the gas is allowed to flow in either direction meaning that the mass flow can be negative and does not influence the performance of our methods. The change of the flow direction does not change the mathematical formulation of algebraic constraints. Therefore, the structure of the system stays
unchanged with respect to the change of the flow direction. This means that we do not require the foreknowledge of the flow direction. 

First, we test two different cases, which corresponds to two different flow direction profiles of the network, cf. Figure~\ref{fig:forked}. Case 1 corresponds to $p_s^{(1)}=p_s^{(2)}=30$ bar, and $q_d=30$ kg/s while case 2 corresponds to $p_s^1=30$ bar, $p_s^2=20$ bar, and $q_d=30$ kg/s. We plot the mass flow at the supply pipe $1$ and $2$ in Figures~\ref{fig:case_1_forked}--\ref{fig:case_2_forked}. 

Figure~\ref{fig:case_1_forked} shows that the mass flow at both supply nodes approaches the steady state after oscillation for a short while, and both input mass flows have a positive sign. This represents that both supply nodes inject gas flow into the network to supply gas to the demand node 6. After changing the operation condition of the network, e.g., changing the pressure at supply nodes, the mass flow is redistributed as shown in Figure~\ref{fig:case_2_forked}. The mass flow at
supply node $10$ becomes negative after a few seconds and remains negative after the network reaches steady state. For this case, the supply pipe at node $10$ acts as a demand pipe since gas flows out of the network there. For both cases, the equality $q_s^1+q_s^2=\sum q_d$ holds which can be easily see by looking at the steady state solution.

We also apply two different cases to a more complicated network given in Figure~\ref{fig:medium_network} to test the robustness of our methods. Case 1 corresponds to $p_s^{55}=p_s^{56}=50.5$ bar, $p_s^{57}=50.8$ bar while case 2 has $p_s^{55}=p_s^{56}=50.5$ bar, and $p_s^{57}=50.0$ bar. The demand of gas at the demand nodes is the same for both cases. We show the mass flow at the pipe that connects node $31$ and $37$, which also connects two sub-networks. The mass flow for the
pipe $31\rightarrow 37$ for different cases is shown in Figure~\ref{fig:flow_pip_31_37}. The initial conditions of the gas network for the simulation of the two different cases are set the same.

\begin{figure}[H]
    \centering
    \subfigure[mass flow at node $1$]
        {\label{fig:qs}
          \includegraphics[width=0.45\textwidth]{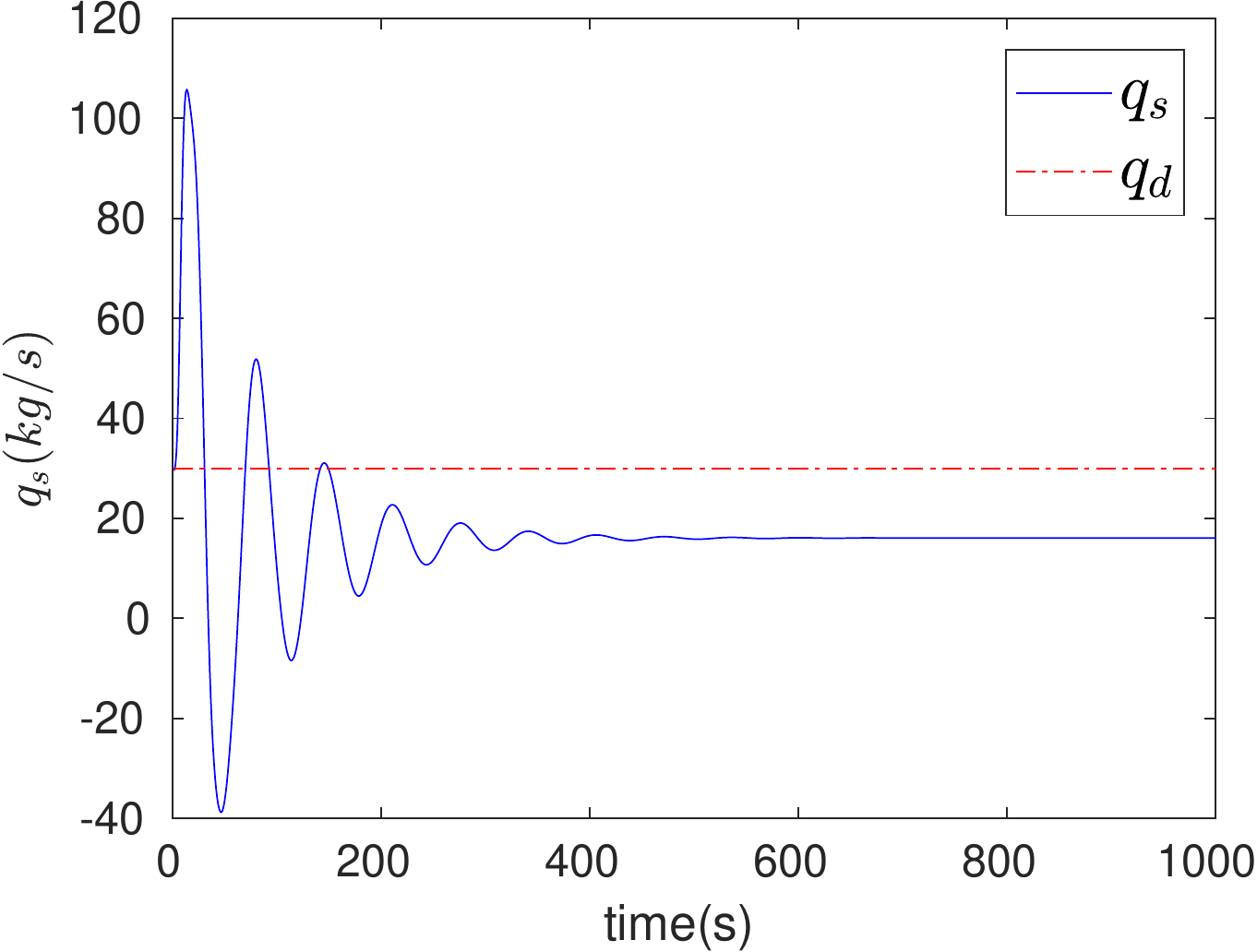}}
    \qquad 
    \subfigure[mass flow at node $10$]
        {\label{fig:qs_zoom}
           \includegraphics[width=0.45\textwidth]{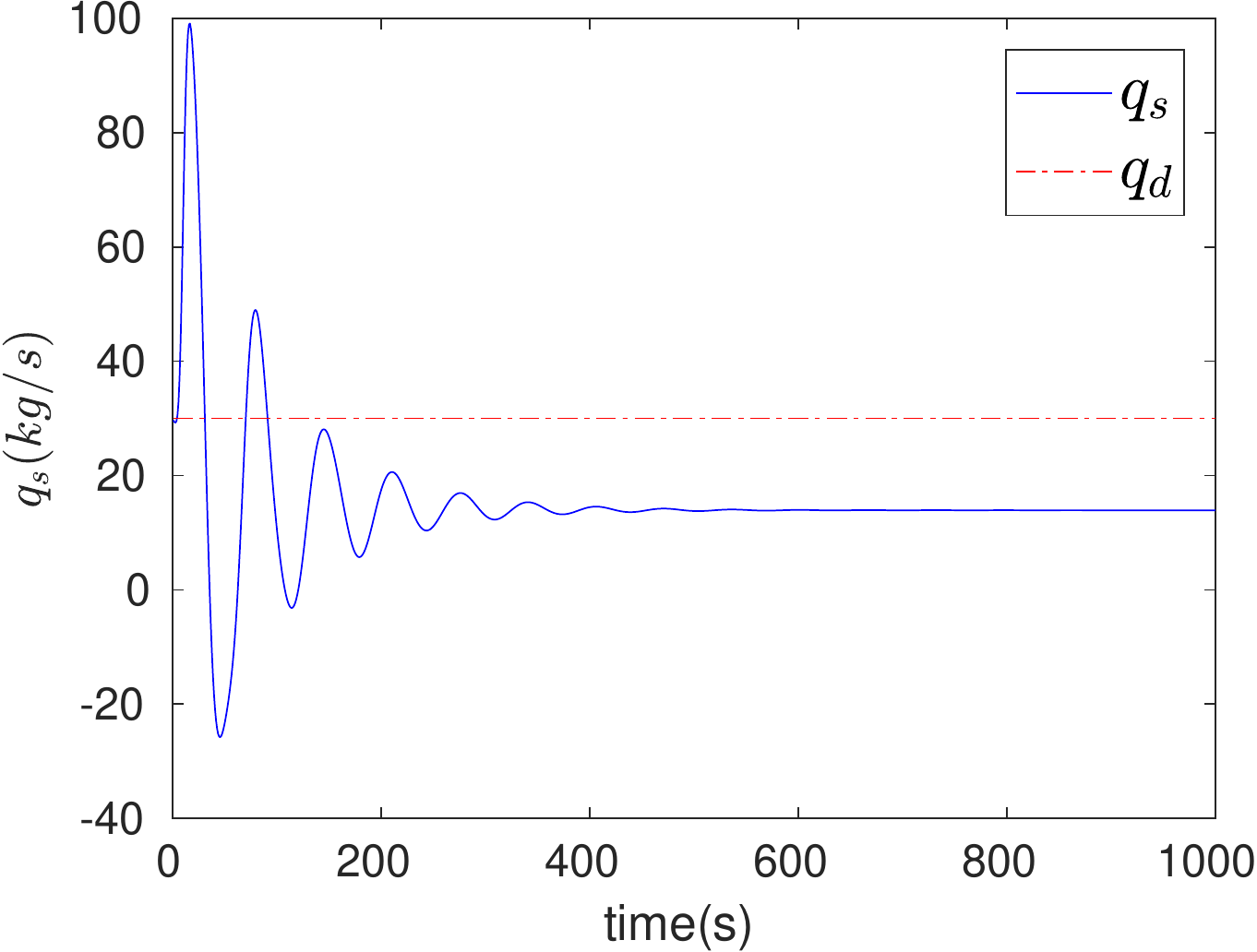}}
         \vspace{-0.3cm}
         \caption{Mass flow at supply nodes for case 1}\label{fig:case_1_forked}
\end{figure}

\vspace{-0.5cm}

\begin{figure}[H]
    \centering
    \subfigure[mass flow at node $1$]
        {\label{fig:qs}
          \includegraphics[width=0.45\textwidth]{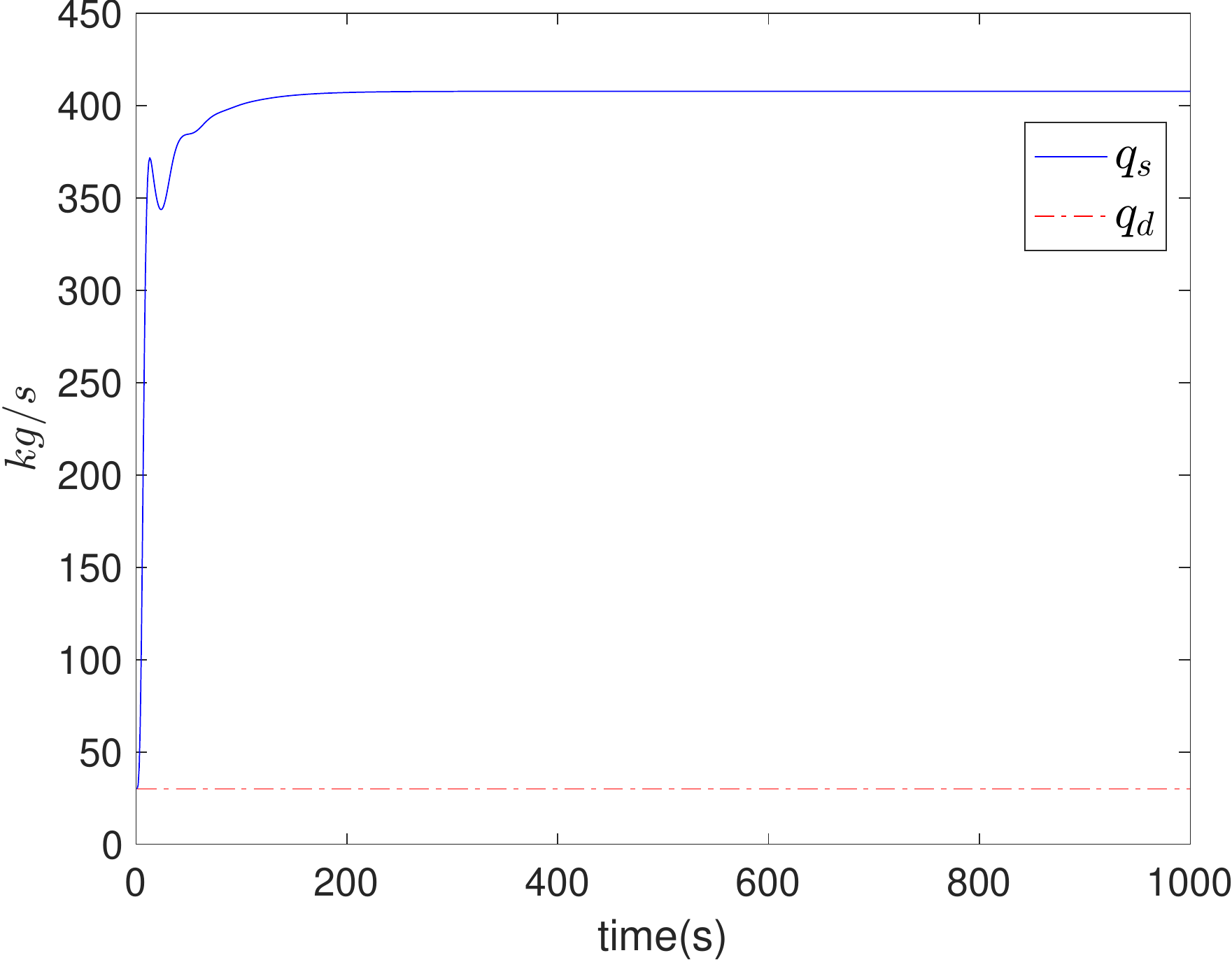}}
    \qquad
    \subfigure[mass flow at node $10$]
        {\label{fig:qs_zoom}
           \includegraphics[width=0.45\textwidth]{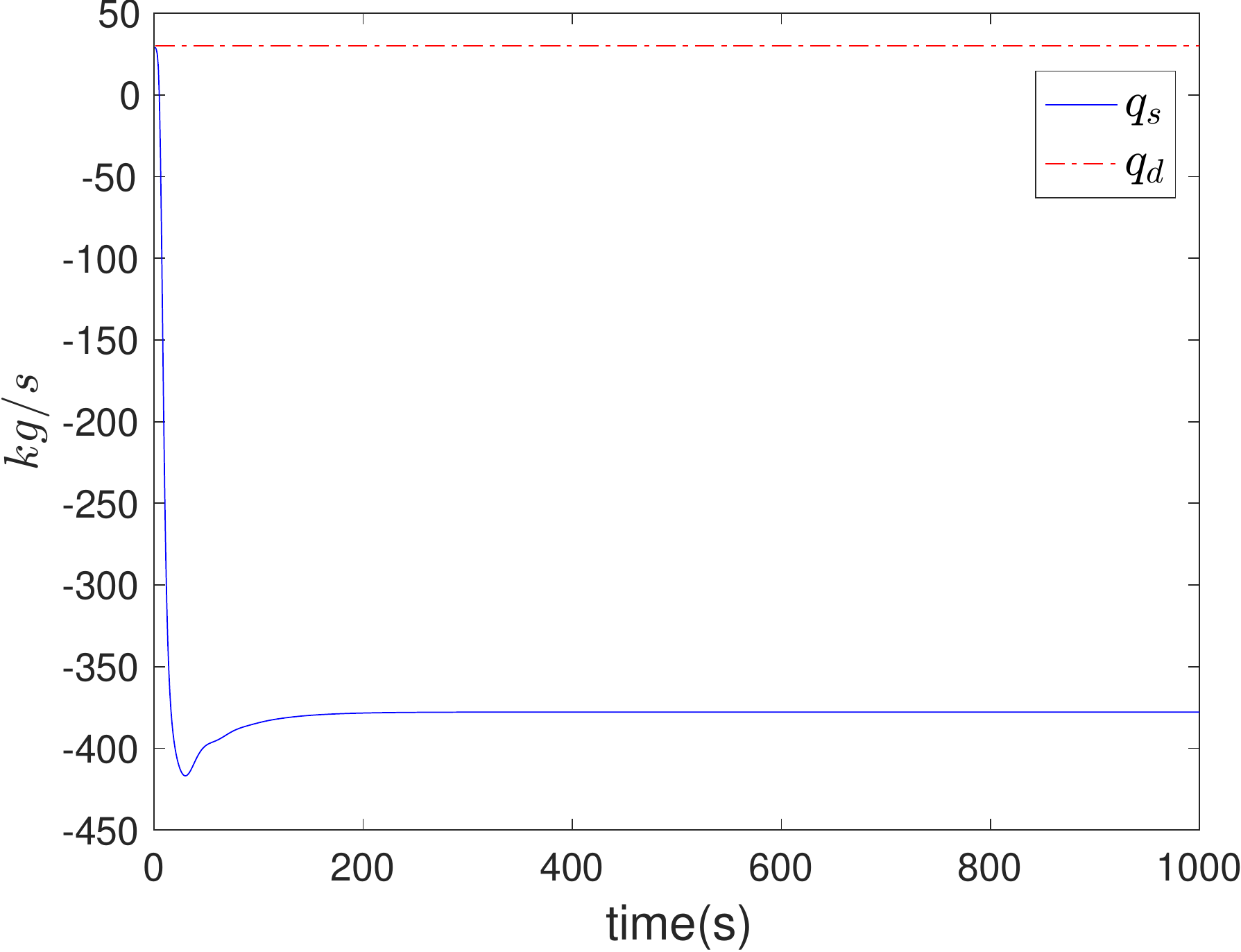}}
    \vspace{-0.3cm}
  \caption{Mass flow at supply nodes for case 2}\label{fig:case_2_forked}
\end{figure}

\vspace{-0.5cm}

\begin{figure}[H]
    \centering
    \subfigure[case 1]
        {\label{fig:case_1}
          \includegraphics[width=0.45\textwidth]{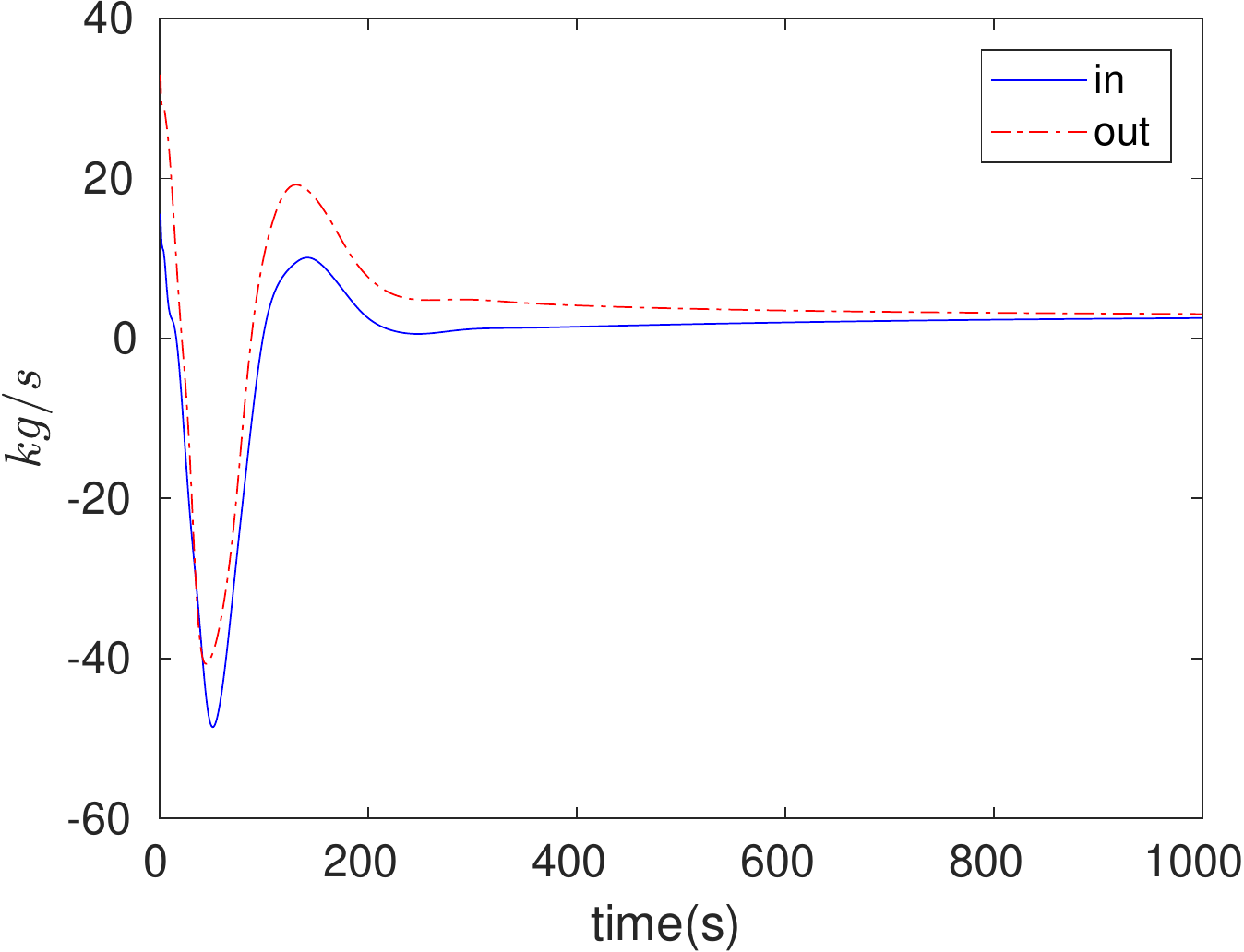}}
    \qquad
    \subfigure[case 2]
        {\label{fig:case_2}
           \includegraphics[width=0.45\textwidth]{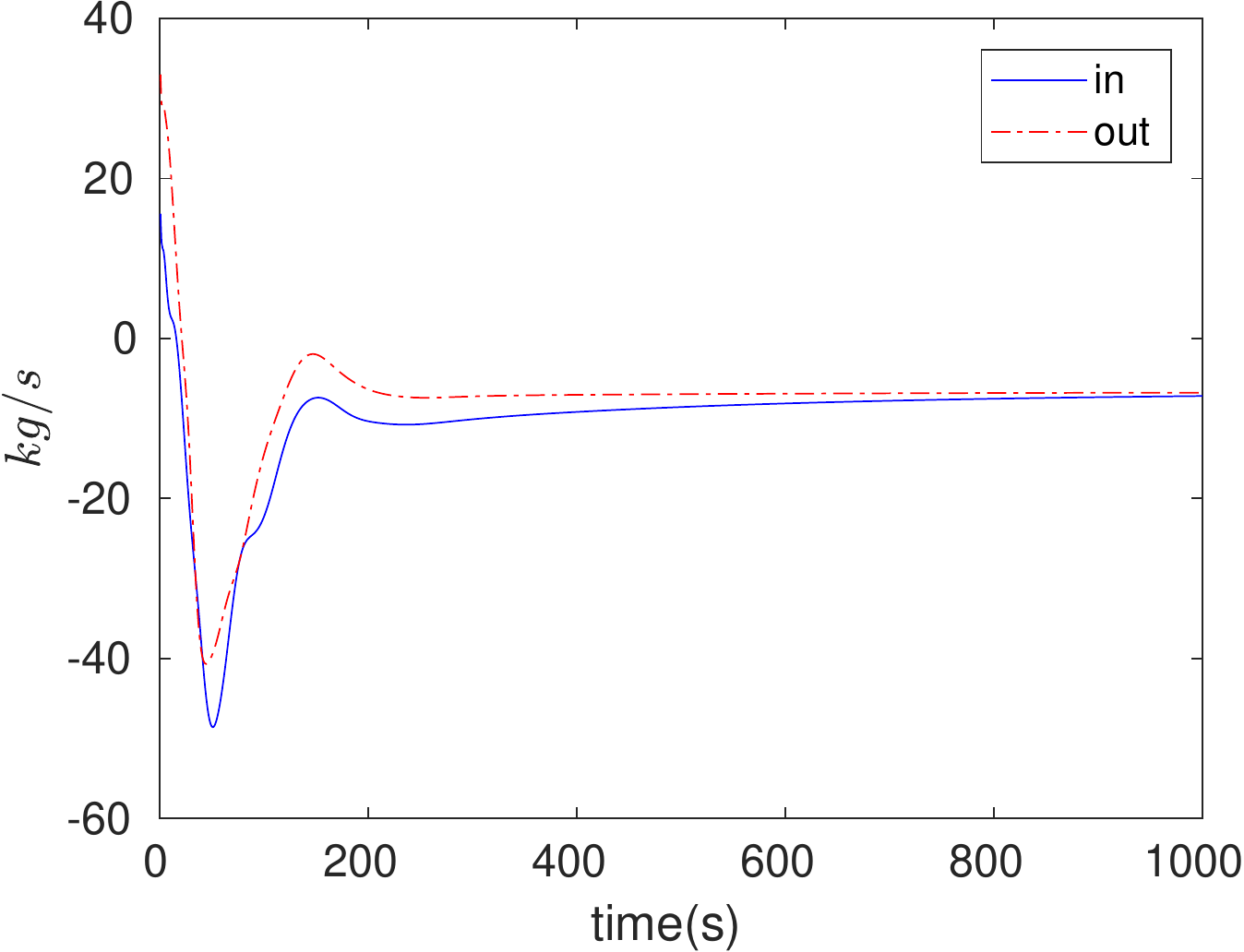}}
     \vspace{-0.3cm}
 \caption{Mass flow for the pipe $31\rightarrow 37$}\label{fig:flow_pip_31_37}
\end{figure}

The simulation results in Figure~\ref{fig:flow_pip_31_37} show that the flow direction at pipe $31\rightarrow 37$ changes for the above two cases. The steady state of the mass flow for the two cases shows that the flow can travel in a direction opposite to the prescribed flow direction, and the inflow at node $31$ is equal to the outflow at node $37$ for the steady state. The imbalance between the inflow and outflow in the transient process is necessary to build the pressure profile of the network.

Computational results in Figure~\ref{fig:case_1_forked}--\ref{fig:flow_pip_31_37} show that the gas can flow in the opposite direction as the directed graph suggests. We do not need to introduce another set of variables and switch to another model when the flow changes direction. 

\subsection{Convergence Comparison}
After the FDF ordering, we apply Algorithm~\ref{alg:newton} to solve the nonlinear equation~\eqref{eqn:k_step} at each time step. Newton's method requires the computation of the Jacobian matrix at each iteration, which is typically expensive. We try to reduce the cost of computing the Jacobian matrix, and we approximate the nonlinear term $\frac{q_m^k |q_m^k|}{q_m^k}$ at the $m$-th Newton iteration of time step $k$ as,
\[
    \frac{q_m^k |q_{m-1}^k|}{p_{m-1}^k}\approx \frac{q_m^k |q_m^k|}{p_m^k}.
\]
This approximation avoids computing the partial derivatives of the nonlinear term, and results in a diagonal coefficient matrix for the linearization. We apply this approximation to solve the nonlinear equation iteratively, and this is the Picard iteration. Here, we study the convergence of both Newton's method and the Picard method for solving~\eqref{eqn:k_step}. We plot the $2$-norm of the nonlinear residual, i.e., $\|F\|_2$ of the first and $50$-est time step at
each Newton and Picard iteration. Both Newton and Picard iteration start with the same initial condition and the time step size for both methods is set as $\tau=1$.

We first study the convergence for the simulation of the network shown in Figure~\ref{fig:medium_network}. We discretize this network using the finite volume method with mesh size $h=50$, and both Newton's method and the Picard method are stopped when $\|F\|_2\leq 10^{-5}$. The results given in Figure~\ref{fig:conv_comp} show that both the Newton and Picard iteration have a fast convergence rate. However, the Picard method takes more than twice the number of iterations to reach the same
stopping criterion. This means that we need to solve more than twice as many linear systems for the Picard method than for
Newton's method, while solving such a linear system is the most time consuming part of such a nonlinear iteration. Moreover, the sizes of the linear systems at each Newton and Picard iteration are the same. The additional cost from the Picard method is much bigger than the cost saved from the simplification of the derivatives computation. This makes the Picard method not as practical as Newton's method for solving such a nonlinear system. 

Next, we test on a bigger network given in Figure~\ref{fig:big_net}. We also discretize this network with the finite volume method with $h=50m$, and we stop both the Newton and Picard iteration when $\|F\|_2 \leq 10^{-4}$. The results are given by Figure~\ref{fig:conv_comp_big}. Again, we observe similar convergence behavior for Newton's method and the Picard method with the convergence results shown in Figure~\ref{fig:conv_comp}. The Picard method needs more than twice the number of
linear system solves than Newton's method. When we need more accurate simulations of a gas network, smaller mesh sizes are necessary increasing the difference in the computational effort.

\begin{figure}[H]
    \centering
    \subfigure[$1$st time step]
        {\label{fig:1st}
          \includegraphics[width=0.45\textwidth]{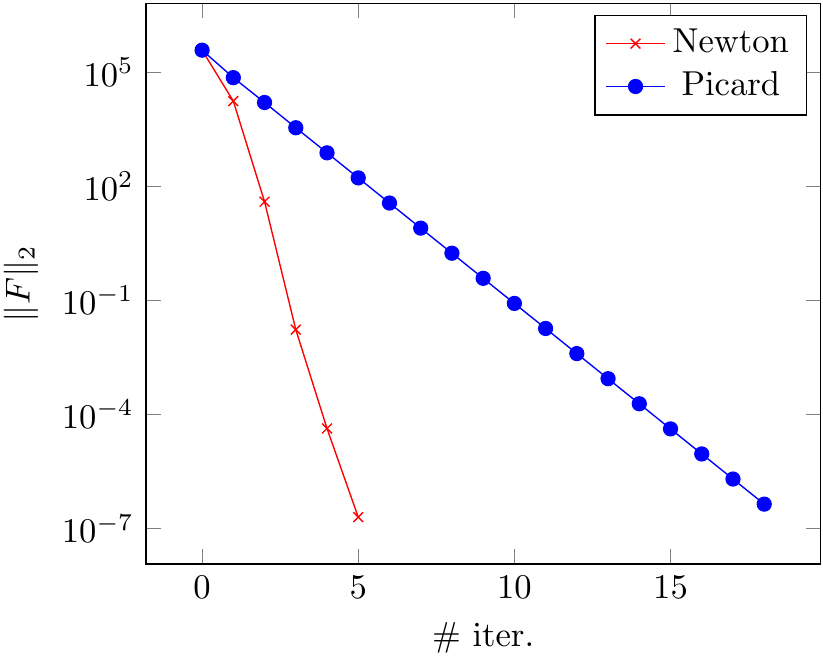}}
    \qquad  
    \subfigure[$50$th time step]
        {\label{fig:50th}
           \includegraphics[width=0.45\textwidth]{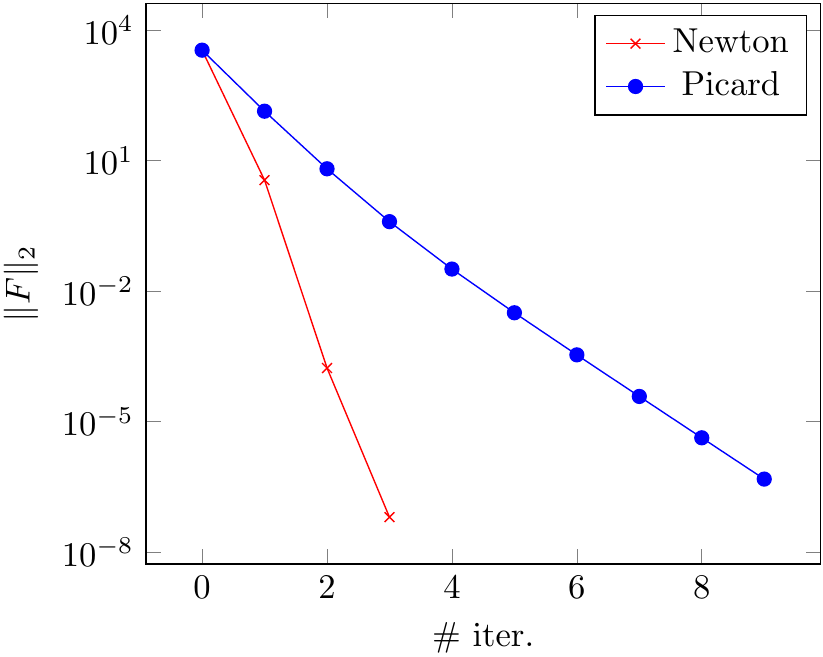}}
      \vspace{-0.3cm}
     \caption{Convergence comparison for network in Figure~\ref{fig:medium_network}}\label{fig:conv_comp}
\end{figure}

\begin{figure}[H]
    \centering
    \subfigure[$1$st time step]
        {\label{fig:1st_big}
          \includegraphics[width=0.45\textwidth]{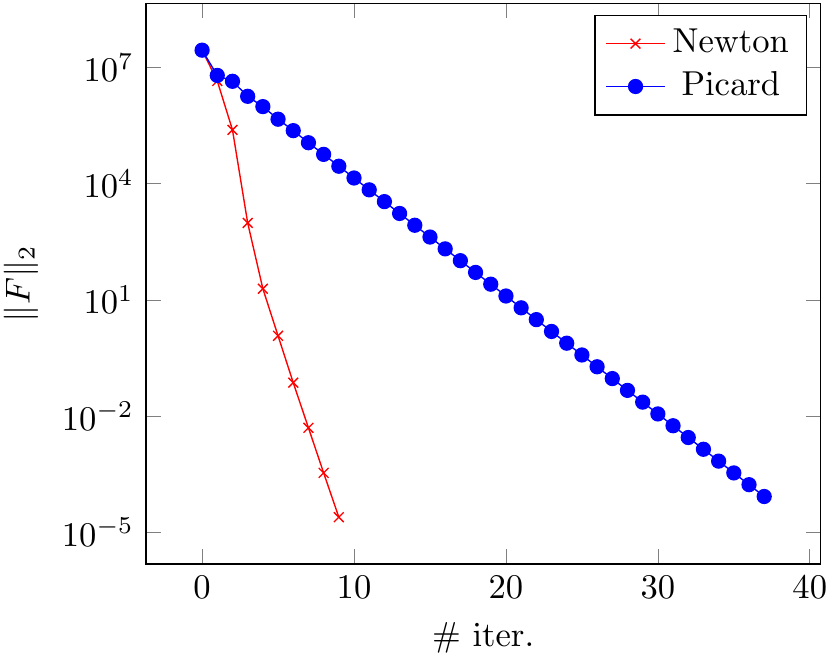}}
    \qquad 
    \subfigure[$50$th time step]
        {\label{fig:50th_big}
           \includegraphics[width=0.45\textwidth]{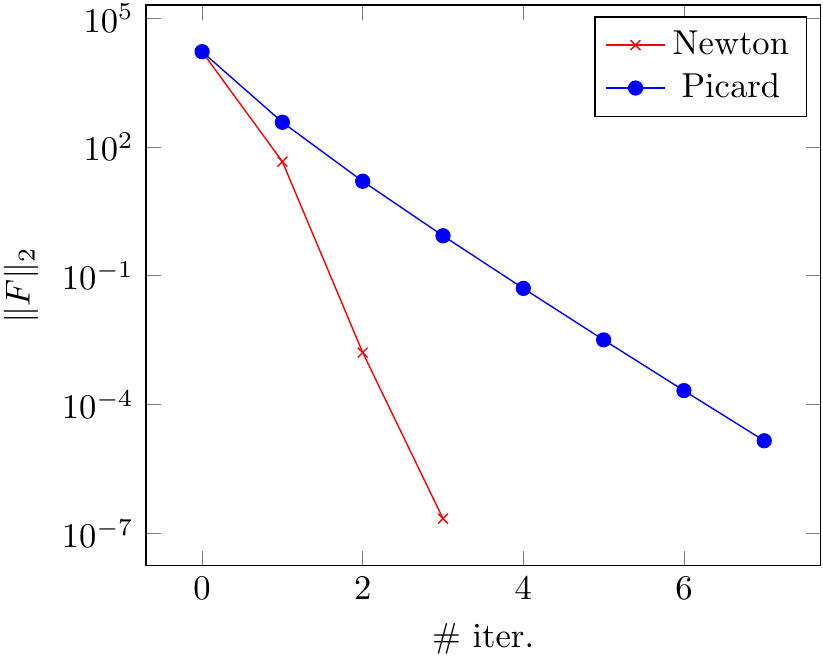}}
           \vspace{-0.3cm}
         \caption{Convergence comparison for network in Figure~\ref{fig:big_net}}\label{fig:conv_comp_big}
\end{figure}
Computational results in Figures~\ref{fig:conv_comp} and~\ref{fig:conv_comp_big} show that for the simulation of gas networks, Newton's method is superior to the Picard method in both the computational complexity and the convergence rate.

\subsection{Preconditioning Performance}
As introduced in the previous section, the biggest challenge for applying Algorithm~\ref{alg:newton} to simulate a gas network lies in the effort spent to solve the linear system at each Newton iteration. For large-scale networks, we need smaller mesh sizes to discretize such networks and this results in larger sizes of the DAEs. Therefore, we need to employ iterative solvers to compute the solution of such a large-scale linear system at each Newton iteration, while preconditioning is
essential to
accelerate the convergence of such iterative solvers. In this part, we study the performance of the preconditioner~\eqref{eqn:p0}.

We test the performance of the preconditioner for the network in Figure~\ref{fig:big_net} using different mesh sizes for the finite volume discretization. At each Newton (outer) iteration, we solve a linear system by applying an (inner) Krylov solver, e.g., the IDR(s) solver~\cite{vanGijzen2011, Sonneveld2008}, and this is called Newton-Krylov method. Note that the Newton-Krylov method is an inexact Newton method, and at each Newton iteration, we apply the IDR(s) method to solve the linear
system up to an accuracy $\varepsilon_{tol}$, i.e.,
\[
    \|F(x_m) + DF(x_m)(x - x_m)\|\leq \varepsilon_{tol}\|F (x_m)\|,
    \]
where $\varepsilon_{tol}$ is related to the forcing term for an inexact Newton's method~\cite{Kelley2003}. Since Newton-Krylov method is inexact, we show its convergence with respect to different tolerances of the Krylov solver, i.e., $\|F\|_2$ with respect to different settings of $\varepsilon_{tol}$. We use the ``true'' residual computed by using a direct method, i.e., the {\tt backslash} operator implemented in MATLAB for comparison. We report the computational results for the FVM discretization with mesh sizes of $50$ and $40$, and the time step size $\tau$ is set to be $1$. For the convergence rate of the inexact Newton method with respect to $\varepsilon_{tol}$, we refer
to~\cite{DemES1982}.

The computational results of the nonlinear residual $\|F\|_2$ in Figure~\ref{fig:nl_res} for two different mesh sizes show that the accuracy of the inner iteration loop can be set relatively low while the convergence of the outer iteration can still be comparable with more accurate inner loop iterations. The convergence properties of the Newton iteration for the first time step are the same if the inner loop is solved accurately, or the inner loop is solved up to an accuracy of $10^{-6}$ or $10^{-4}$. If the inner loop is solved up to an accuracy of $10^{-3}$, only one more Newton iteration is needed. Moreover, the convergence behavior of the Newton iteration for different inner loop solution tolerances are the same for the $10$-th time step. If lower inner loop accuracy is used, less computational effort is needed. This reduces the computational complexity. The number of IDR($4$)
iterations for different inner loop tolerances are reported in Figure~\ref{fig:idr_iter_1st}--\ref{fig:idr_iter_10th}.

\begin{figure}[H]
    \centering
    \subfigure[$h=50$, $1$st time step]
        {\label{fig:nl_res_1st_50}
          \includegraphics[width=0.45\textwidth]{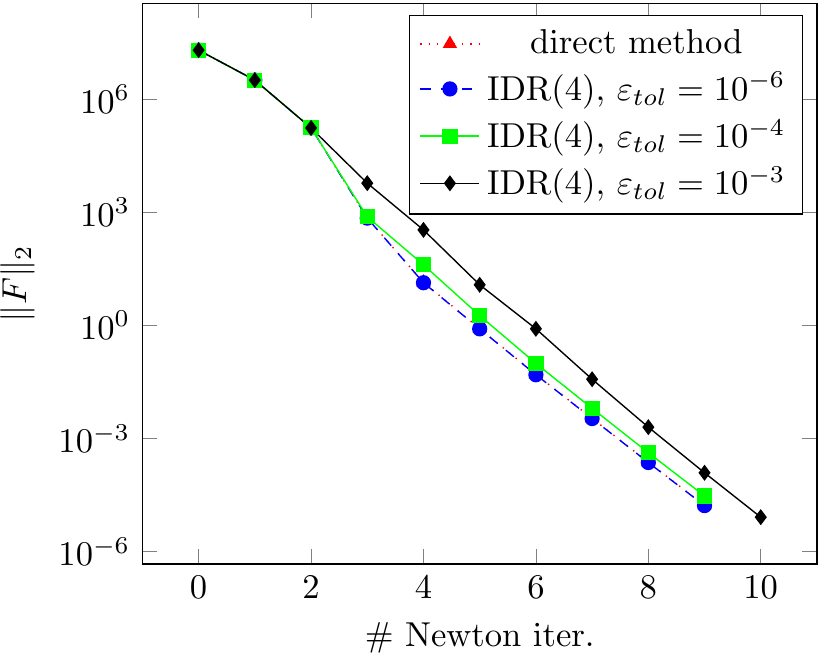}}
    \quad
     \subfigure[$h=50$, $10$-th time step]
        {\label{fig:nl_res_10th_50}
          \includegraphics[width=0.45\textwidth]{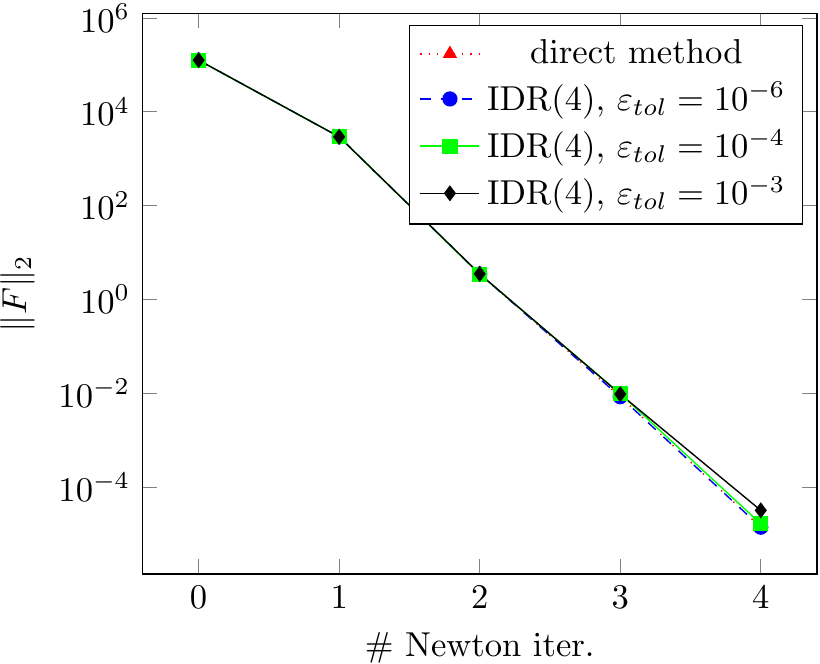}}

    \subfigure[$h=40$, $1$st time step]
        {\label{fig:nl_res_1st_40}
           \includegraphics[width=0.45\textwidth]{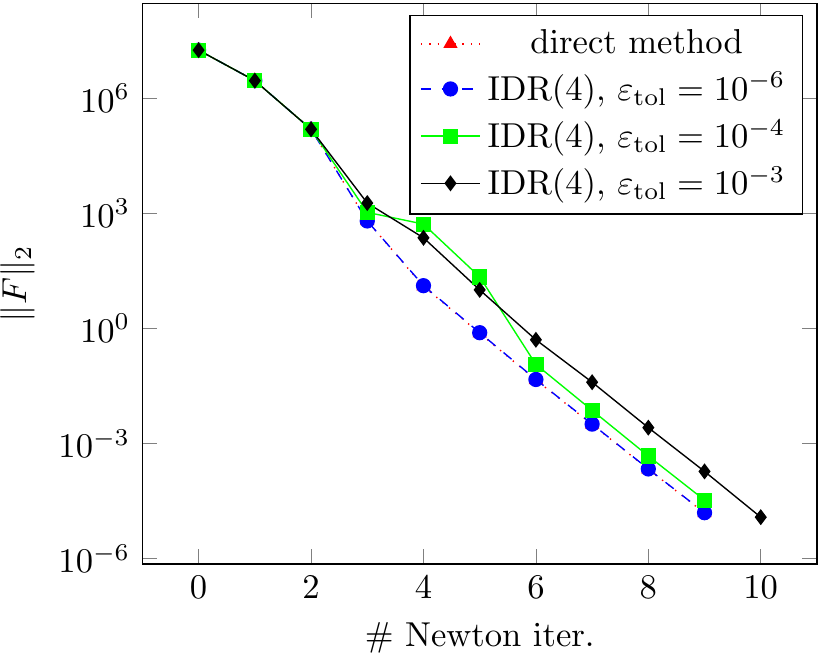}}
     \quad
     \subfigure[$h=40$, $10$-th time step]
        {\label{fig:nl_res_10th_40}
           \includegraphics[width=0.45\textwidth]{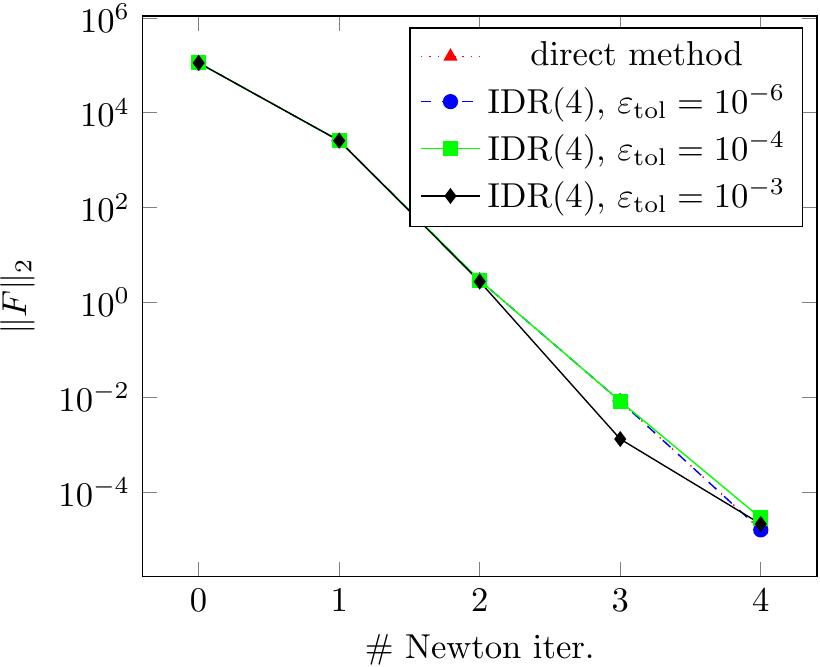}}
     \vspace{-0.3cm}
         \caption{Nonlinear residual at the first and tenth time step}\label{fig:nl_res}
\end{figure}

\begin{figure}[H]
    \centering
    \subfigure[$h=50$]
        {\label{fig:iter_1st_50}
          \includegraphics[width=0.48\textwidth]{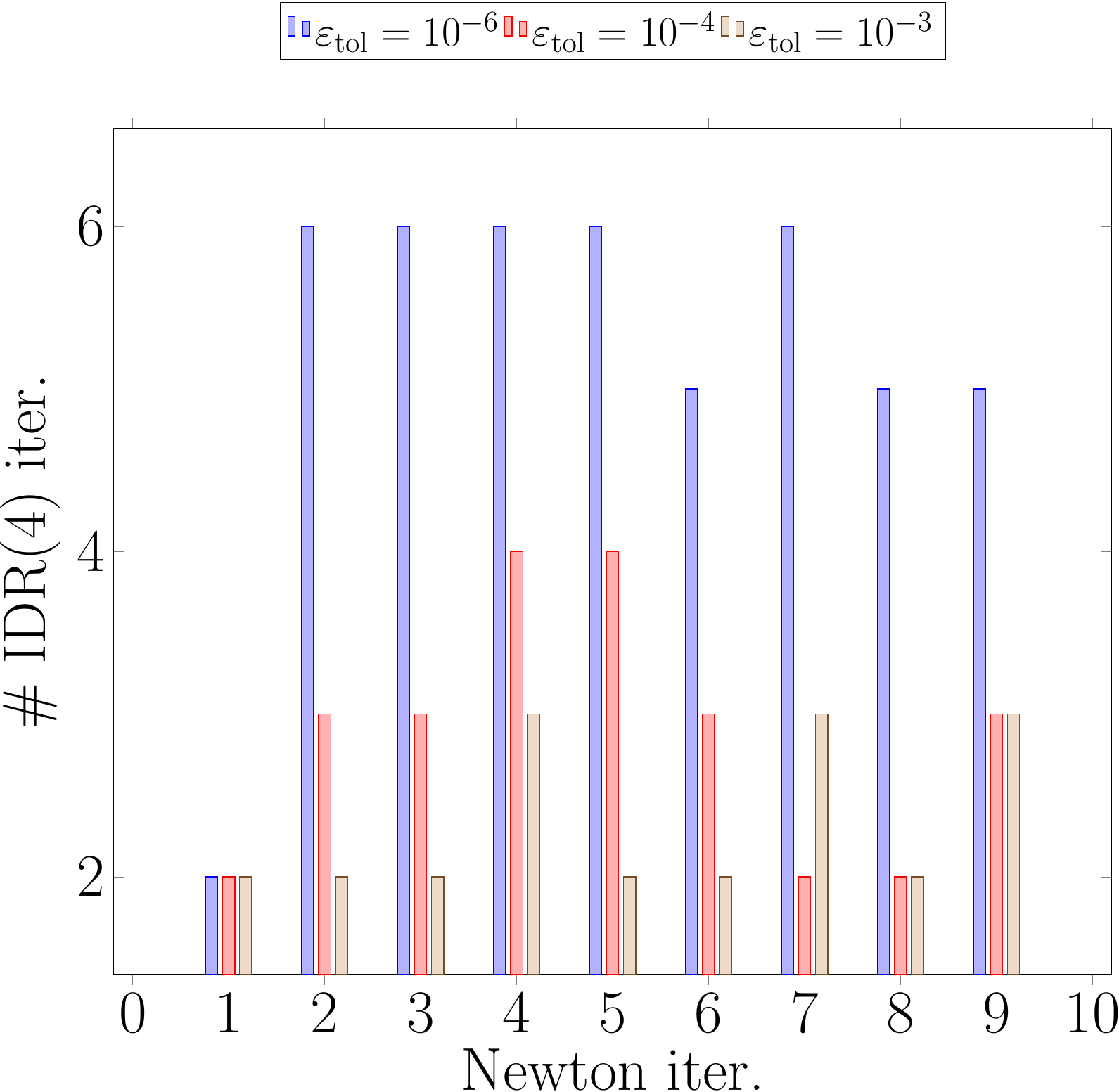}}
    \subfigure[$h=40$]
        {\label{fig:iter_1st_40}
           \includegraphics[width=0.48\textwidth]{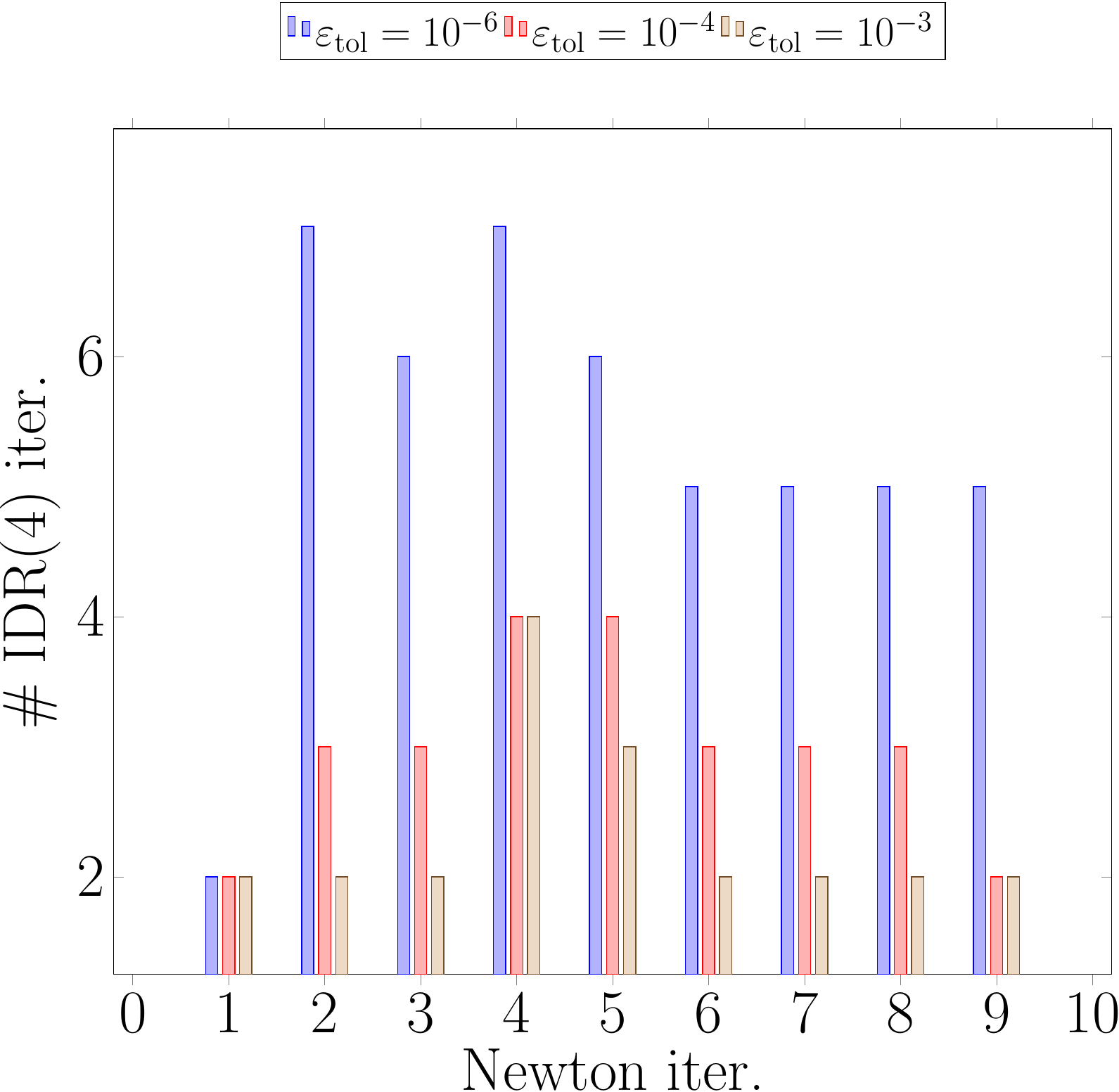}}
    \vspace{-0.3cm}
       \caption{Number of IDR($4$) iterations at the first time step}\label{fig:idr_iter_1st}
\end{figure}

The computational results in Figure~\ref{fig:iter_1st_50} show that the total number of IDR($4$) iterations ($47$) for $\varepsilon_{\text{tol}}=10^{-6}$ is almost twice the total number of IDR($4$) iterations ($24$) for $\varepsilon_{\text{tol}}=10^{-3}$. This demonstrates that the computational work for the first time step can be reduced to almost $50\%$ since the most time consuming part inside each Newton iteration is the IDR($4$) solver. Similar results are shown by Figure~\ref{fig:iter_1st_40}.
As the system gets closer to steady state, less Newton iterations are needed, and the IDR($4$) solver also needs less iterations, as shown in Figure~\ref{fig:idr_iter_10th}. At this stage, IDR($4$) with $\varepsilon_{\text{tol}}=10^{-3}$ still needs less work than the IDR($4$) with $\varepsilon_{\text{tol}}=10^{-6}$ but is no longer as significant.

\begin{figure}[H]
    \centering
    \subfigure[$h=50$]
        {\label{fig:iter_10th_50}
          \includegraphics[width=0.46\textwidth]{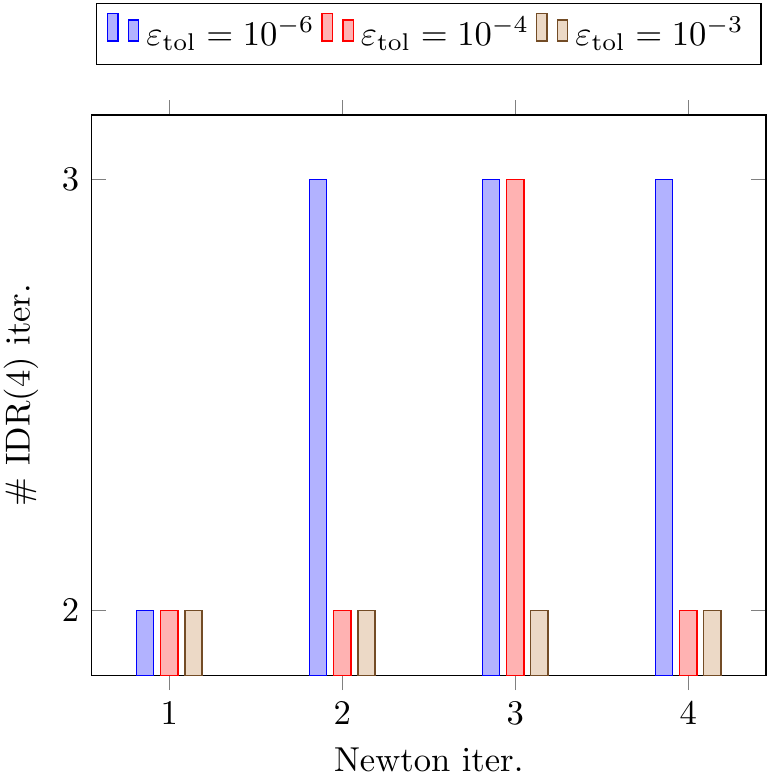}}
    \qquad
    \subfigure[$h=40$]
        {\label{fig:iter_10th_40}
           \includegraphics[width=0.46\textwidth]{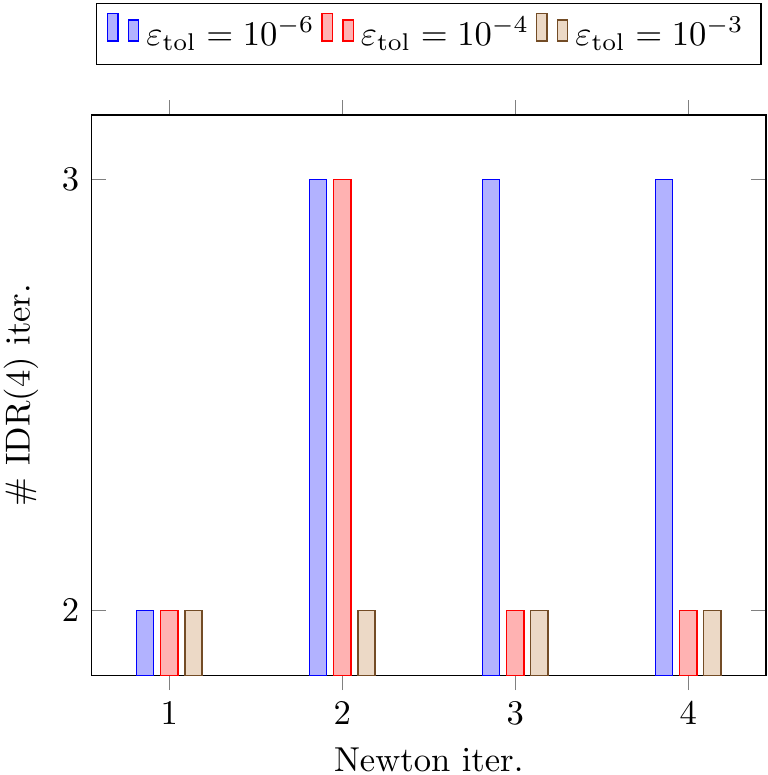}}
    \vspace{-0.3cm}
    \caption{Number of IDR($4$) iterations at the $10$-th time step}\label{fig:idr_iter_10th}
\end{figure}
We have already observed that the performance of our modeling is robust and convergence of the preconditioned Krylov solver happens quickly. Next, we report the computational time for computing the preconditioner $P_1$ and applying the preconditioned IDR($4$) solver of the first Newton step for the first time step for different FVM discretization mesh sizes. We solve the Jacobian system up to an accuracy of $10^{-4}$, and the computational results are given by Table~\ref{tab:timing}. Here ``-'' represents running out of memory, $\# D_F$ represents the size of the Jacobian system, $t_{S^1}$ denotes the time to compute the Schur complement
in $P_1$, and all the time are measured in seconds.

\begin{table}[H]
    \centering
    \caption{\footnotesize Computational time for the 1st Newton iteration}
    \label{tab:timing}
    \vspace{0.3cm}
    \begin{tabular}{ccccc}
        \toprule
        $h$  & $\# D_F$ & $t_{S^1}$  & IDR($4$) & {\tt backslash}  \\
        \midrule
        40   & 1,03e+05 & $3.85$      & 0,25     & 0,13 \\ 
        \addlinespace
        20  & 2,01e+05 & $8,12$      & 0,52     & 0,36 \\ 
        \addlinespace
        10  & 3,97e+05 & $17,84$     & 1,06     & 1,18 \\ 
        \addlinespace
        5   & 7,91e+05 & $38,44$     & 2,13     & 1054,62 \\ 
        \addlinespace
        2.5 & 1,58e+06 &$81,42$     & 4,34     & - \\ 
        \bottomrule
    \end{tabular}
\end{table}

The computational results in Table~\ref{tab:timing} show the advantage of our preconditioner in solving the large-scale Jacobian system over the direct solver. The time to solve the preconditioned system using the IDR(4) solver scales linearly with the system size, and is much smaller than the time to apply the direct solver when the  mesh sizes are smaller than $20$. For large-scale Jacobian systems, the direct solver either takes up too much CPU time or fails to solve the Jacobian system due to running out of memory. For
smaller Jacobian systems, the direct solver shows the advantage over preconditioned Krylov solvers. This is primarily because there is a big overhead when applying the preconditioned IDR(4) solver while the {\tt backslash} operator is highly optimal for smaller systems. The time to compute the Schur complement in preconditioner $P_1$ scales almost linearly with the system sizes.

%% file: conclusions.tex
In this paper, we studied the modeling and simulation of pipeline gas networks. We applied the finite volume method (FVM) to discretize the incompressible isothermal Euler equation, and compared it with the finite difference method (FDM). Numerical results show the advantage of the FVM over the FDM. To model gas networks, we introduced the SDF gas network, which represents the topology of the network interconnection and reduces the size of the algebraic constraints of the
resulting differential algebraic equation (DAE) compared with current research. To simulate such a DAE system, we proposed the \textit{direction following} (DF) ordering of the edges of the SDF network. Through such an DF ordering, we exploited the structure of the system matrix and proposed an efficient preconditioner to solve the DAE. Numerical results show the advantage of our algorithms.